\documentclass[11pt,reqno]{amsart}
\usepackage[a4paper,margin=27mm]{geometry}
\usepackage[T1]{fontenc}
\usepackage{lmodern,amsmath,amssymb,amsthm,mathtools,mathrsfs,microtype}
\usepackage[hidelinks]{hyperref}
\usepackage{enumitem}
\allowdisplaybreaks[2]
\usepackage{xcolor}
\hypersetup{
colorlinks=true,
linkcolor=blue,
citecolor=blue,
urlcolor=blue
}
\numberwithin{equation}{section}
\newtheorem{theorem}{Theorem}[section]
\newtheorem{lemma}[theorem]{Lemma}
\newtheorem{assumption}[theorem]{Assumption}
\newtheorem{corollary}[theorem]{Corollary}
\theoremstyle{definition}
\newtheorem{definition}[theorem]{Definition}

\theoremstyle{remark}

\newcommand{\E}{\mathcal E}
\newcommand{\HH}{\mathscr H}
\newcommand{\Op}{\mathscr A}
\newcommand{\norm}[1]{\left\|#1\right\|}

\newcommand{\essinf}{\operatorname*{ess\,inf}}
\newcommand{\esssup}{\operatorname*{ess\,sup}}

\title[A Viscoelastic Wave Equation]{A Viscoelastic Wave Equation with Nonlinear Damping and Supercritical Variable-Exponent Logarithmic Sources}
\author[M.  Liao]{Menglan Liao}
\address{School of Mathematics, Hohai University, Nanjing, 210098, China \&
Laboratory of Mathematical Modeling and Intelligent Computing for Water Systems, Hohai University, Nanjing, 211100, China}
 \email{liaoml@hhu.edu.cn}
 
\author[M. Zhang]{Mingxue Zhang}
\address{Institute for Math and AI, Wuhan University, Wuhan, 430072, China. }
 \email{zhangmingxuex@163.com}
\thanks{This work was supported by the National Natural Science Foundation of China (12401290) and the Natural Science Foundation of Jiangsu Province (BK20230946).}
\date{}
\begin{document}
\begin{abstract}
We investigate a viscoelastic wave equation with nonlinear variable-exponent damping and a variable-exponent logarithmic source, allowing supercritical growth. The main difficulties arise from the failure of the source to map the energy space into $L^2(\Omega)$, the sign-changing logarithmic potential, and the lack of Sobolev control of the displacement for supercritical
damping. Under an exponent condition linking the source and damping exponents, we prove local existence, an energy identity, and a continuation criterion.  
We prove finite-time blow-up for nontrivial solutions with
nonpositive shifted initial energy. We also establish a
blow-up criterion based on a sufficiently large initial
displacement--velocity pairing, with a threshold depending
on the positive part of the shifted initial energy, and
derive an explicit upper bound for the lifespan. This
criterion allows us to construct blow-up data at every
prescribed level of the original energy. For initial data in a suitable stable region, solutions exist globally.  For damping within the Sobolev range, a general relaxation condition yields energy decay estimates that include exponential and polynomial rates. Beyond this range, a time-weighted displacement estimate and a splitting of the memory integral give logarithmic decay.
\end{abstract}
\maketitle

\section{Introduction}
Wave equations with damping and sources describe the competition
between energy dissipation and nonlinear amplification. Depending
on the initial data and the relative growth of these terms,
solutions may exist globally and decay or become unbounded in
finite time. In a viscoelastic medium, the stress depends on both
the current strain and its past history, while the relaxation
kernel describes how the influence of earlier deformations
decreases. Mathematical foundations and stability results for
viscoelasticity are developed in
\cite{Dafermos1970,FabrizioMorro1992}. Evolution formulations
of equations with memory are studied in
\cite{FabrizioGiorgiPata2010}, and nonlinear wave equations
combining memory, frictional damping, and supercritical sources
are considered in \cite{GRSTT,GRSDecay2018}.

In this paper, we consider the initial boundary value problem
\begin{equation}
\label{eq1}
\begin{cases}
u_{tt}-\Delta u+\displaystyle\int_0^t g(t-s)\Delta u(s)\,ds
+|u_t|^{m(x)-2}u_t=|u|^{p(x)-2}u\log|u|,
&x\in\Omega,\ t>0,\\
u=0,&x\in\partial\Omega,\ t>0,\\
u(0)=u_0,\quad u_t(0)=u_1,&x\in\Omega,
\end{cases}
\end{equation}
where $\Omega\subset\mathbb R^n$, $n\ge3$, is a bounded domain
with $C^2$ boundary. The logarithmic source is defined by
continuity at zero. The memory kernel $g$ is positive,
nonincreasing, and integrable, with total mass less than one.
The measurable exponents $m(x)$ and $p(x)$ describe spatial
variation in the damping and source growth. 
The convolution starts at time zero, so the memory
accumulates during the evolution. No displacement history
before the initial time is prescribed.

Variable exponents provide a way to describe nonlinear
responses whose growth changes with position. One motivation
comes from electrorheological fluids, whose constitutive behavior
depends on an applied electric field. Their modeling and
mathematical analysis are developed in the monograph of
R\r{u}\v{z}i\v{c}ka \cite{Ruzicka2000}. Related questions
concerning heterogeneous materials and nonlinear variational
problems arise in the homogenization theory of Zhikov
\cite{Zhikov1987}. These applications have motivated the
development of Lebesgue and Sobolev spaces with variable
exponents. Their basic properties were studied by
Kov\'a\v{c}ik and R\'akosn\'ik
\cite{KovacikRakosnik1991}, and a comprehensive treatment
is given in \cite{DHHR}. Although the Luxemburg norm remains
homogeneous, the integral $\int_\Omega |u|^{p(x)}\,dx$
cannot generally be expressed as a fixed power of that norm.
Energy estimates must therefore account for both the lower
and upper bounds of the exponent.
Evolution equations with nonstandard growth are systematically
studied by Antontsev and Shmarev
\cite{AntontsevShmarev2015}, including questions of existence,
uniqueness, localization, and blow-up. For hyperbolic problems,
Antontsev \cite{Antontsev2011} investigated blow-up for wave
equations with a space- and time-dependent exponent in the
principal operator. Messaoudi and Talahmeh
\cite{MessaoudiTalahmeh2017} established blow-up results
for quasilinear wave equations with variable-exponent
nonlinearities, including negative-energy solutions and
suitable positive-energy data. Messaoudi, Talahmeh, and
Al-Smail \cite{MessaoudiTalahmehAlSmail2017} used the
Faedo--Galerkin method to prove existence and uniqueness
under suitable exponent assumptions for a wave equation
with variable-exponent frictional damping and source,
and also obtained finite-time blow-up. The stability and lifespan of solutions have been studied
for several related models. Li, Guo, and Liao
\cite{LiGuoLiao2020} obtained decay and asymptotic stability
for a quasilinear hyperbolic equation with variable growth
in the principal operator, damping, and source. Their proof
uses energy estimates and an integral inequality of Komornik
type \cite{Komornik1994}. For a viscoelastic Kirchhoff equation with nonlinear
inertia, Liao, Guo, and Zhu \cite{LiaoGuoZhu2020} derived
upper blow-up time bounds for suitable positive-energy data
and lower lifespan bounds in different exponent ranges.
Liao and Tan \cite{LiaoTan2023} established blow-up,
lifespan, global existence, and decay results for a Petrovsky
equation with strong damping. Their high-energy blow-up
result concerns linear frictional damping, while the strong
damping provides additional spatial control of the velocity.

Supercritical sources present another difficulty because
they need not be locally Lipschitz from $H_0^1(\Omega)$
into $L^2(\Omega)$. The usual perturbation argument in
the energy space is then unavailable, and the integrability
provided by nonlinear damping becomes important.
Bociu and Lasiecka
\cite{BociuLasiecka2008U,BociuLasiecka2010} established
uniqueness, continuous dependence, and local Hadamard
well-posedness for wave equations with supercritical
interior and boundary sources. Guo and Rammaha
\cite{GRWave} developed a local and global well-posedness
theory for coupled wave equations using monotone operators
and nonlinear semigroups. Guo, Rammaha, Sakuntasathien,
Titi, and Toundykov \cite{GRSTT} treated a viscoelastic
equation with infinite past history. Their history
formulation and source approximation provide the framework
for our local construction. To the best of our knowledge, the present work gives the
first systematic results on local existence for viscoelastic wave equations with supercritical
variable-exponent logarithmic sources.

The competition between damping and supercritical sources
plays a central role in determining whether solutions exist
globally and how their energy evolves. When the source dominates the dissipation for suitable initial
data, solutions may blow up in finite time. Sufficiently strong
dissipation can lead to global existence and energy decay.
Bociu and Lasiecka \cite{BociuLasiecka2008B} established
blow-up criteria under conditions on the initial data and
the relative growth of the source and damping terms.
Guo and Rammaha \cite{GRBlow2013} proved blow-up for
negative-energy solutions of coupled wave equations.
In \cite{GRDecay2013}, they used the potential-well method
to identify a stable set of initial data yielding global
existence and exponential or algebraic energy decay. They
also obtained blow-up results for suitable data with
nonnegative energy. For a viscoelastic problem with
history, Guo, Rammaha, and Sakuntasathien \cite{GRBlow} established blow-up when the source exponent exceeds both the damping and memory-growth exponents, for negative-energy data and for positive-energy data with sufficiently large quadratic energy. Subsequently, they 
\cite{GRSDecay2018} incorporated the memory term into the
potential-well framework and established global existence
and decay for data in its stable region. Related results
for a piezoelectric beam system with nonlinear damping and
supercritical sources were obtained by Liao and Feng
\cite{LiaoFeng2026}.

The logarithmic source requires additional care in both
the blow-up and stability arguments. Its contribution
$f(x,u)u=|u|^{p(x)}\log|u|$ is negative for $0<|u|<1$,
and its potential is not homogeneous. Consequently,
the usual relation between a power source and its
potential does not apply directly. Liao
\cite{LiaoLog2022} studied a constant-exponent logarithmic
source in a viscoelastic equation with strong damping
and obtained high-energy blow-up and lifespan estimates.
Very recently, we studied in \cite{LZ} a viscoelastic wave
equation with a variable-exponent logarithmic source and
strong damping. We established blow-up at arbitrary energy
levels and derived lifespan estimates. Our analysis uses a
pointwise correction in the source--potential inequality,
whose integral is absorbed into a constant shift of the
energy. In the present work, we adapt this approach to a
problem with frictional damping. The resulting estimates
control the displacement--velocity pairing through the
frictional dissipation while retaining the memory
contribution.

Energy decay for viscoelastic equations depends on the
relaxation kernel and the available dissipation.
Classical sufficient conditions include
\[
g'(t)\le-\kappa g(t),
\quad
g'(t)\le-\kappa g(t)^q,
\quad \kappa>0,\quad q>1.
\]
Under appropriate additional hypotheses, these conditions
lead to exponential or polynomial energy bounds.
Alabau-Boussouira, Cannarsa, and Sforza
\cite{AlabauCannarsaSforza2008} developed a unified
approach to such estimates for second-order evolution
equations with memory and semilinear sources.
Messaoudi \cite{Messaoudi2008} considered
\[
g'(t)\le-\xi(t)g(t),
\]
where $\xi$ is positive and nonincreasing. The resulting
energy estimates are expressed through
$\int_0^t\xi(s)\,ds$, allowing more general rates.
Alabau-Boussouira and Cannarsa
\cite{AlabauCannarsa2009} developed a convexity method
for sharp decay estimates under nonlinear kernel
conditions of the form
\[
g'(t)\le-\mathcal H(g(t)).
\]
Mustafa \cite{Mustafa2018} combined the time-dependent
coefficient and nonlinear relaxation function through
\[
g'(t)\le-\xi(t)\mathcal H(g(t)),
\]
where $\xi$ is positive and nonincreasing, and
$\mathcal H$ is linear or strictly increasing and
strictly convex near the origin. For the linear
viscoelastic wave equation, this condition gives
general decay estimates that include exponential
and polynomial bounds. 
Pata \cite{Pata2010} studied exponential stability
for kernels allowing almost flat portions, while
Conti and Pata \cite{ContiPata2020} obtained decay
for abstract linear viscoelasticity without imposing
differential inequalities on the kernel.
We adopt the relaxation framework of Mustafa and
address the additional estimates required by the
nonlinear damping and logarithmic source.

The damping range also plays an important role in energy
decay. When $m_+\le2^*=2n/(n-2)$, Sobolev embedding
provides the spatial integrability needed to estimate the
displacement and the memory term in the damping space
$L^{m(x)}(\Omega)$. When $m_+>2^*$, finite energy alone
does not guarantee this integrability. Thus, supercritical
damping and supercritical sources present distinct
limitations of the Sobolev embedding.
Li and Guo \cite{LiGuo2024} obtained logarithmic
decay for a quasilinear wave equation with
variable-exponent diffusion and sources and
constant-exponent supercritical damping.
Li and Li \cite{LiLi2024} derived logarithmic decay
for a viscoelastic wave equation with variable
exponents under exponentially decaying kernels.
Their arguments use displacement estimates that
replace the unavailable spatial embedding. For a wave equation with infinite memory and
constant-exponent power nonlinearities, Liao
\cite{LiaoMemory2025} established weighted decay,
including a logarithmic example, under a differential
kernel condition and a divergence condition on a
combined time weight. For an equation without memory
or a source, Haraux and Tebou
\cite{HarauxTebou2025} obtained polynomial decay
for strong solutions in specified dimension ranges
and logarithmic decay for weak solutions with
additional displacement integrability.
Qin, Liao, and Zhang \cite{QinLiaoZhang2025}
and Yan, Guo, and Zhu \cite{YanGuoZhu2026}
used weighted displacement estimates to prove
logarithmic decay for Petrovsky and Timoshenko-type
equations with variable-exponent supercritical
damping. These results motivate our treatment
of damping beyond the Sobolev range.

Our first result establishes local energy solutions under
the strict exponent conditions
\[
(p_+-1)m_-'<2^*,
\quad
m_-'=\frac{m_-}{m_--1},
\quad
p_+<2^*.
\]
The condition $(p_+-1)m_-'<2^*$ ensures that the source
belongs to the dual damping space after its logarithmic
factor is bounded by a slightly higher power. The condition
$p_+<2^*$ allows us to choose this higher power within the
Sobolev range. We first truncate the source and solve the
resulting evolution problem using a maximal monotone
operator. Uniform energy estimates give a common existence
interval, and a difference energy estimate allows us to
pass to the original source. We then use temporal
regularization to prove the exact energy identity for every
energy solution. Finally, we establish a continuation
criterion that takes account of the full memory history
when a solution is restarted.

Under additional assumptions on the relative growth of the
source and damping exponents and on the total mass of the
memory kernel, we prove finite-time blow-up for nontrivial
data with nonpositive shifted energy. If the initial
shifted energy is zero, the energy identity shows that it
becomes strictly negative at every positive time. We can
then apply the negative-energy argument starting at any
such time while retaining the memory accumulated since
time zero. We also establish a blow-up criterion based on
a sufficiently large initial displacement--velocity
pairing, which applies regardless of the sign of the
initial shifted energy. This criterion allows us to
construct smooth blow-up data at every prescribed level
of the original energy. For global existence, we identify
an explicit stable region bounded by a potential barrier.
Its invariance yields coercive estimates for the velocity,
the displacement gradient, and the memory energy. These
bounds keep the solution within the continuation criterion
and imply global existence.

For stable solutions with $m_+\le2^*$, the kernel condition
used by Mustafa~\cite{Mustafa2018} yields general
inverse-function decay estimates. Exponential and
polynomial decay follow for particular choices of the
relaxation function. When $m_+>2^*$, the energy bound alone
does not give the displacement integrability needed to
estimate the damping pairing. Assuming additionally that
$u_0\in L^{m(x)}(\Omega)$, we obtain logarithmic decay
through a time-weighted estimate for the displacement.
In the memory term, we split the history according to its
age. On bounded age intervals, the kernel condition
controls the memory contribution by the dissipation in the
energy identity. For the remaining tail, the integrability
of the kernel makes its mass small when the splitting
point is chosen sufficiently large. Monotonicity of the
time weight then allows this tail to be absorbed after
integration in time. The resulting weighted energy
integrability yields the logarithmic decay estimate.

Section~\ref{sec:results} presents the assumptions
and main results. Section~\ref{sec:localproof}
proves local existence, the energy identity,
and continuation.
Sections~\ref{sec:blowproof} and~\ref{sec:globalproof}
establish the blow-up criteria and global existence,
respectively.
Sections~\ref{sec:decayproof} and~\ref{sec:superproof}
treat decay within and beyond the damping Sobolev range.

\section{Assumptions and main results}\label{sec:results}
Throughout the paper, $C>0$ denotes a generic constant that may
change from line to line, whereas $C_1,C_2,\ldots$ denote
nonnegative constants that remain fixed within the corresponding
argument. Dependence on parameters is indicated when relevant.

We set \(V=H_0^1(\Omega)\) and \(H=L^2(\Omega)\). The symbol \(\|\cdot\|_r\) refers to the standard norm of \(L^r(\Omega)\), and we equip $V$ with the norm \(\|u\|_V=\|\nabla u\|_2\). By Poincar\'e's inequality, there exists a constant \(C_p>0\) such that
\[\|u\|_2\le C_p\|\nabla u\|_2,\quad \forall u\in V.\]
We write $(\cdot,\cdot)$ for the $H$ inner product. For a measurable exponent $r$ with 
\[1<r_-:=\essinf_{x\in\Omega}r(x)\le r(x)\le r_+:=\esssup_{x\in\Omega}r(x)<\infty,\] 
define
\[\rho_r(u)=\int_\Omega |u|^{r(x)}\,dx\in[0,\infty]\]
and
\[L^{r(x)}(\Omega)=\{u\text{ is measurable on }\Omega:\rho_r(u)<\infty\}\]
equipped with the Luxemburg norm
\[\|u\|_{r(x)}=\mathrm{inf}\Big\{\lambda>0:\rho_r\Big(\frac u\lambda\Big)\leq1\Big\}.\] 
When this notation is adopted for \(Q_T=\Omega\times(0,T)\), the exponent is understood to depend solely on $x$.  We write $r'(x)=r(x)/(r(x)-1)$ and $r_-'=r_-/(r_--1)$.
For a real number $a$, let $[a]_+=\max\{a,0\}$. 

We make use of several consequences following from the definition of the Luxemburg norm.
For further details, we refer to \cite{DHHR}. For $1<r_-\le r(x)\le r_+<\infty$,
\begin{equation}
\label{eq2}
\min\{\|u\|_{r(x)}^{r_-},\|u\|_{r(x)}^{r_+}\}
\le\int_\Omega|u|^{r(x)}dx
\le\max\{\|u\|_{r(x)}^{r_-},\|u\|_{r(x)}^{r_+}\}.
\end{equation}
The variable-exponent H\"older's inequality is
\begin{equation}
\label{eq3}
\int_\Omega|vw|dx\le2\|v\|_{r(x)}\|w\|_{r'(x)}.
\end{equation}
Both statements also hold on $Q_T$. In particular, convergence of the modular
to zero is equivalent to convergence of the norm to zero. The 
Young's inequality, with an arbitrary fixed $\varepsilon>0$, gives
\begin{equation}
\label{eq4}
|vw|\le\varepsilon |v|^{r(x)}+C_\varepsilon |w|^{r'(x)},
\end{equation}
where $C_\varepsilon$ can be chosen uniformly for $r(x)\in[r_-,r_+]$.

The following assumptions will be imposed throughout the paper: 
\begin{assumption}\label{ass1}
\end{assumption}
\begin{enumerate}
\item The exponents are measurable and satisfy
\[2<p_-\le p(x)\le p_+<2^*,\quad 2\le m_-\le m(x)\le m_+<\infty,\quad (p_+-1)\frac{m_-}{m_--1}<2^*.\]
  \item The memory kernel function $g$ satisfies
  \[g\in W^{1,1}_{\mathrm{loc}}([0,\infty))\cap L^1(0,\infty),\quad g>0,\quad g'\le0~\text{a.e.},
\quad \ell=1-\int_0^\infty g(s)\,ds>0.\]
\end{enumerate}

Define
\[
X_m=V\cap L^{m(x)}(\Omega),
\]
equipped with the norm
\[
\|u\|_{X_m}=\|\nabla u\|_2+\|u\|_{m(x)}.
\]
The space $X_m$ is reflexive.


Define
\[f(x,u)=|u|^{p(x)-2}u\log|u|,	\quad b(x,u)=|u|^{m(x)-2}u,\]
\[F(x,u)=\frac{|u|^{p(x)}}{p(x)}\log|u|-\frac{|u|^{p(x)}}{p(x)^2},\quad F(x,0)=0.\]
The associated energy functional reads
\begin{equation}\label{eq5}
\begin{aligned}
E(t)={}&\frac12\|u_t(t)\|_2^2
+\frac12\left(1-\int_0^t g(s)\,ds\right)\|\nabla u(t)\|_2^2\\
&+\frac12\int_0^t g(t-s)\|\nabla u(t)-\nabla u(s)\|_2^2\,ds
-\int_\Omega F(x,u(t))\,dx.
\end{aligned}
\end{equation}

\begin{definition}[Energy solution]
\label{def:energy}
Let $T>0$ and $(u_0,u_1)\in V\times H$.
A function $u$ is called an energy solution on $[0,T]$ if
\[
u\in C([0,T];V),
\quad
u_t\in C([0,T];H)\cap L^{m(x)}(Q_T),
\]
where $u_t$ denotes the distributional time derivative
of $u$, and $u(0)=u_0\in V$, $u_t(0)=u_1\in H.$
Moreover, for every $z\in X_m$, 
\begin{equation}
\label{eq6}
\begin{split}
&\frac{d}{dt}(u_t(t),z)
+(\nabla u(t),\nabla z)
-\int_0^t g(t-s)(\nabla u(s),\nabla z)\,ds\\
&\qquad+\int_\Omega b(x,u_t(t))z\,dx
=\int_\Omega f(x,u(t))z\,dx
\end{split}
\end{equation}
in $\mathcal D'(0,T)$.
\end{definition}

\begin{theorem}[Local existence and continuation]\label{thm:local}
Suppose that Assumption \ref{ass1} holds.
Given any initial pair $(u_0,u_1)\in V\times H$, there exists some $T>0$ together with an energy solution $u$ defined on $[0,T]$ for problem \eqref{eq1} satisfying the initial conditions $(u(0),u_t(0))=(u_0,u_1)$.
$T$ may be taken uniformly positive for initial data lying within bounded subsets of $V\times H$.

Each local energy solution possesses a maximal extension onto an interval $[0,T_{\max})$, with $T_{\max}\in(0,\infty]$.
For arbitrary $0<T<T_{\max}$, we have
$
f(\cdot,u)\in L^{m'(x)}(Q_T),
$
and the energy identity
\begin{equation}\label{eq7}
\begin{aligned}
E(t)
&+\int_s^t\int_\Omega
|u_\tau(x,\tau)|^{m(x)}\,dx\,d\tau+\frac12\int_s^t
g(\tau)\|\nabla u(\tau)\|_2^2\,d\tau\\
&+\frac12\int_s^t\int_0^\tau
[-g'(\tau-r)]
\|\nabla u(\tau)-\nabla u(r)\|_2^2\,dr\,d\tau
=E(s)
\end{aligned}
\end{equation}
is valid for all $0\le s\le t\le T$.

Whenever $T_{\max}<\infty$, the blow-up alternative
\begin{equation}\label{eq8}
\limsup_{t\uparrow T_{\max}}
\bigl(\|u(t)\|_V+\|u_t(t)\|_2\bigr)=\infty
\end{equation}
holds true.
\end{theorem}

We now investigate conditions under which energy solutions blow up in
finite time. 
To investigate the blow-up criterion, we need to impose an additional assumption.
 \begin{assumption}\label{ass2}
 \[m_+<p_-,\quad \ell>\frac{1}{(p_--1)^2}.
\]
 \end{assumption}
Following our previous work \cite{LZ}, we use a
source--potential inequality with an additive
correction, which is absorbed into a constant
shift of the energy. For $q\ge p_-$, define
\[
c_{p_-}(q)=
\begin{cases}
\displaystyle\frac{q-p_-}{q^2}\exp\!\left(-\frac{q}{q-p_-}\right),&q>p_-,\\[2mm]
0,&q=p_-.
\end{cases}
\]
Lemma~\ref{correctionlemma} states its optimal pointwise property.
To absorb the additive term in the source--potential estimate, we use $C_*:=\int_\Omega c_{p_-}(p(x))\,dx$
and define the \emph{shifted energy}
\begin{equation*}
\E(t)=E(t)+\frac{C_*}{p_-}.
\end{equation*}
The function $c_{p_-}$ is continuous on $[p_-,p_+]$, so
\begin{equation*}
0\le C_*=\int_\Omega c_{p_-}(p(x))\,dx\le |\Omega|\max_{q\in[p_-,p_+]}c_{p_-}(q)<\infty.
\end{equation*}
Since the shift $C_*/p_-$ is independent of time, $\E'=E'$ wherever the derivatives exist.

Throughout the following three results, assume that
Assumption~\ref{ass1} and \ref{ass2} hold. 
Fix
\begin{equation}\label{eq9}
0<\alpha<
\min\left\{
\frac{p_--2}{2p_-},
\frac{p_--m_+}{p_-(m_+-1)}
\right\},
\quad
r=\frac1{1-\alpha}>1.
\end{equation}
For each initial pair $(u_0,u_1)\in V\times H$,
let $u$ be any maximal energy solution supplied by
Theorem~\ref{thm:local}, with maximal existence time
$T_{\max}$. 

\begin{theorem}[Blow-up for nonpositive shifted energy]
\label{thm:negative}
Let Assumption~\ref{ass1} and \ref{ass2} hold. Suppose that
\[
\mathcal E(0)\le0,
\quad
(u_0,u_1)\ne(0,0).
\]
Then $T_{\max}<\infty$ and
\[
\limsup_{t\uparrow T_{\max}}
\bigl(\|u(t)\|_V+\|u_t(t)\|_2\bigr)=\infty.
\]
More precisely, the following lifespan estimates hold.
\begin{enumerate}[label=\textup{(\roman*)}]
\item
If $\mathcal E(0)<0$, set $h_0=-\mathcal E(0)$.
There exist positive constants $\varepsilon,c,A$,
chosen as in Section~\ref{sec:negative} and possibly
depending on the initial data, such that
\[
Y_0=h_0^{1-\alpha}+\varepsilon(u_0,u_1)>0
\]
and
\[
T_{\max}
\le\frac{A}{\varepsilon c(r-1)}Y_0^{1-r}.
\]

\item
If $\mathcal E(0)=0$, then
\[
\mathcal E(\tau)<0
\quad\text{for every }0<\tau<T_{\max}.
\]
For each such $\tau$, set $h_\tau=-\mathcal E(\tau)$.
There exist positive constants
$\varepsilon_\tau,c_\tau,A_\tau$, chosen as in
Section~\ref{sec:negative}, such that
\[
Y_\tau=h_\tau^{1-\alpha}
+\varepsilon_\tau(u(\tau),u_t(\tau))>0
\]
and
\[
T_{\max}
\le\tau+
\frac{A_\tau}{\varepsilon_\tau c_\tau(r-1)}
Y_\tau^{1-r}.
\]
\end{enumerate}
\end{theorem}

The next criterion imposes no sign restriction on
the shifted initial energy. Instead, it requires a sufficiently
large initial displacement--velocity pairing.

\begin{theorem}[An initial-pairing blow-up criterion]
\label{thm:high}
Let Assumption~\ref{ass1} and \ref{ass2} hold. There exist positive constants $\varepsilon,c,A,B$,  chosen as in Section~\ref{sec:high}
depending only on the structural parameters and the
fixed choice of $\alpha$ in \eqref{eq9}, such that
the following statement holds.

If $(u_0,u_1)\in V\times H$ satisfies
\begin{equation}\label{eq10}
Y_0:=1+\varepsilon(u_0,u_1)>
\left[
\frac{2AB}{c}
\bigl(1+[\mathcal E(0)]_+\bigr)
\right]^{1/r},
\end{equation}
then every corresponding maximal energy solution
satisfies
\begin{equation}\label{eq11}
T_{\max}
\le\frac{2A}{\varepsilon c(r-1)}
Y_0^{1-r}
\end{equation}
and the norm blow-up conclusion \eqref{eq8}.
\end{theorem}

The initial-pairing criterion can be satisfied by
smooth data on every prescribed level of the original
energy $E(t)$.

\begin{theorem}[Blow-up data at every energy level]
\label{thm:every}
Let Assumption~\ref{ass1} and \ref{ass2} hold. For every $R_0\in\mathbb R$, there exist initial data
$
(u_0,u_1)\in C_c^\infty(\Omega)\times C_c^\infty(\Omega)
$
such that
\[
E(0)=R_0
\]
and \eqref{eq10} holds with the constants
from Theorem~\ref{thm:high}.
Consequently, every corresponding maximal energy
solution blows up in the sense of
\eqref{eq8}, and its maximal existence
time satisfies \eqref{eq11}.
\end{theorem}

The following results concern global solutions and their decay.  Choose $\sigma>0$ such
that $p_++\sigma<2^*$, and let $S_p$ satisfy
$\|u\|_{p_++\sigma}\le S_p\|\nabla u\|_2$. Define
\begin{equation}
\label{eq12}
 R=\left(\frac{e\sigma\ell}{2S_p^{p_++\sigma}}\right)^{1/(p_++\sigma-2)},
 \quad d_0=\frac{\ell(p_--1)}{2p_-}R^2.
\end{equation}


\begin{theorem}[Global existence]\label{thm:global}
Let Assumption \ref{ass1} hold and
\begin{equation}\label{eq13}
 \|\nabla u_0\|_2<R,\quad E(0)<d_0,
\end{equation}
then every maximal energy solution supplied by Theorem~\ref{thm:local} is global. Moreover,
\begin{equation}\label{eq14}
c_E\widehat{E}(t) \le E(t)\le C_E\widehat{E}(t).
\end{equation}
 Here 
\[\widehat{E}(t)=\|u_t(t)\|_2^2+\|\nabla u(t)\|_2^2
 +\int_0^t g(t-s)\|\nabla u(t)-\nabla u(s)\|_2^2\,ds.\]
\end{theorem}

For the decay estimates, it suffices to consider $E(0)>0$.
Indeed, \eqref{eq14} and the energy identity
\eqref{eq7} imply
\[
0\le E(t)\le E(0),\quad t\ge0.
\]
If $E(0)=0$, then $E(t)\equiv0$, and
\eqref{eq14} together with Poincar\'e's
inequality yields $u\equiv0$. The decay estimates
therefore hold trivially. For the following decay results we use the following kernel condition
used by Mustafa~\cite{Mustafa2018}, formulated for the locally absolutely
continuous kernel in Assumption~\ref{ass1}.

\begin{assumption}\label{assadd}
There exist a positive nonincreasing function
$\xi\in C^1([0,\infty))$ and a positive function
$\mathcal H\in C^1((0,\infty))$ such that
\begin{equation}\label{eq15}
g'(t)\le-\xi(t)\mathcal H(g(t))
\quad\text{for almost every }t>0.
\end{equation}
Either $\mathcal H(s)=ks$ for some $k>0$, or, for some
$r_0\in(0,g(0)]$, $\mathcal H$ is strictly increasing and strictly convex
on $(0,r_0]$, belongs to $C^2((0,r_0])$, and extends to $[0,r_0]$
with $\mathcal H(0)=\mathcal H'(0)=0$.
\end{assumption}

In the nonlinear case of Assumption \ref{assadd}, define
\begin{equation}\label{eq16}
\Psi(s)=\int_s^{r_0}\frac{d\tau}{\tau\mathcal H'(\tau)},
\quad 0<s\le r_0.
\end{equation}
Then $\Psi$ is strictly decreasing from $(0,r_0]$ onto $[0,\infty)$.

\begin{theorem}[Decay with an upper Sobolev bound on damping]
\label{thm:convex}
Under the hypotheses of Theorem~\ref{thm:global}, suppose that 
$m_+\le2^*$ and Assumption  \ref{assadd} hold. 
Then the following decay estimates hold.
\begin{enumerate}[label=\textup{(\roman*)}]
\item If $\mathcal H(s)=ks$ for some $k>0$, then
\[
E(t)\le CE(0)\exp\left(-c\int_0^t\xi(s)\,ds\right),
\quad t\ge0.
\]

\item If $\mathcal H$ is nonlinear, fix $t_*>0$ such that
$g(t_*)\le r_0$. Then
\begin{equation}\label{eq17}
E(t)\le CE(0)\Psi^{-1}\left(c\int_{t_*}^t\xi(s)\,ds\right),
\quad t\ge t_*.
\end{equation}
If, in addition, $\int_0^\infty\xi(s)\,ds=\infty$, then
\begin{equation}\label{eq18}
E(t)\le CE(0)\Psi^{-1}\left(c\int_0^t\xi(s)\,ds\right),
\quad t\ge0.
\end{equation}
\end{enumerate}
Here $c,C>0$ are constants that may vary between estimates.
They depend only on the fixed parameters and, in case
\textup{(ii)}, may also depend on $t_*$.
\end{theorem}

\begin{corollary}[Decay for linear and power relaxation]
\label{thm:decay}
Under the hypotheses of Theorem~\ref{thm:convex}, suppose
that the function $\mathcal H$ in Assumption~\ref{assadd}
is given by $\mathcal H(s)=s^q$ for some $1\le q<2$.
Then there exist constants $c,C>0$ such that, for all
$t\ge0$,
\begin{equation}\label{eq19}
E(t)\le
\begin{cases}
\displaystyle
CE(0)\exp\left(-c\int_0^t\xi(s)\,ds\right),
& q=1,\\[6pt]
\displaystyle
CE(0)\left(1+c(q-1)\int_0^t\xi(s)\,ds\right)^{-1/(q-1)},
& 1<q<2.
\end{cases}
\end{equation}
\end{corollary}

The next result removes the upper Sobolev restriction on the damping
exponent.

\begin{theorem}[Decay without an upper Sobolev bound on damping]
\label{thm:super}
Under the hypotheses of Theorem~\ref{thm:global}, suppose that 
$m_+>2$, $u_0\in L^{m(x)}(\Omega)$ and Assumption  \ref{assadd} hold. 
Set
\begin{equation}\label{eq20}
\gamma=\max\left\{\frac{m_+-2}{2},\frac1{m_--1}\right\}.
\end{equation}
Then there is a constant $\widehat{C}>0$ such that
\begin{equation}\label{eq21}
E(t)\le E(0)\left(
\frac{1+\widehat{C}}{1+\log(1+t)}\right)^{1/(\gamma+1)},
\quad t\ge0.
\end{equation}
\end{theorem}

\begin{corollary}[Constant supercritical damping in three dimensions]
\label{cor:log}
Assume the hypotheses of Theorem~\ref{thm:global} and
Assumption~\ref{assadd}. Let $n=3$, $m(x)\equiv m>6$,
and $u_0\in V\cap L^m(\Omega)$. Then there is a constant $C>0$ such that
\[
E(t)\le CE(0)[1+\log(1+t)]^{-2/m},\quad t\ge0.
\]
\end{corollary}

\section{Local existence and continuation}\label{sec:localproof}
We first establish growth bounds for the logarithmic source
and Lipschitz estimates for its truncations. These estimates
allow us to construct approximate solutions and pass to the limit.
\begin{lemma}[Growth estimates for the logarithmic source]\label{lem:growth}
Choose
\[
0<\sigma<\min\left\{2^*-p_+,\,
\frac{2^*}{m_-'}-(p_+-1)\right\},
\] there is a constant $C>0$,
independent of $x$ and $s$, such that
\[
\begin{aligned}
|f(x,s)|
&\le C\bigl(1+|s|^{p_+-1+\sigma}\bigr),\\
|\partial_s f(x,s)|
&\le C\bigl(1+|s|^{p_+-2+\sigma}\bigr),\\
|F(x,s)|
&\le C\bigl(1+|s|^{p_++\sigma}\bigr)
\end{aligned}
\]
for almost every $x\in\Omega$ and every $s\in\mathbb R$.
Moreover, $f(x,\cdot)\in C^1(\mathbb R)$, with
\[
\partial_s f(x,s)
=
|s|^{p(x)-2}
\bigl[1+(p(x)-1)\log|s|\bigr]
\quad\text{for }s\ne0,
\quad
\partial_s f(x,0)=0.
\]
\end{lemma}

\begin{proof}
By direct differentiation, one obtains 
\begin{align*}
\frac{d}{dr}(r^{-\sigma}\log r)
&=r^{-\sigma-1}(1-\sigma\log r),\quad r\ge1,\\
\frac{d}{dr}(-r^a\log r)
&=r^{a-1}(-a\log r-1),\quad 0<r\le1,\quad a>0.
\end{align*}
The derivatives change from positive to negative at $e^{1/\sigma}$ and
$e^{-1/a}$, respectively. Therefore, 
\begin{equation}
\label{eq22}
\sup_{r\ge1}r^{-\sigma}\log r=\frac1{e\sigma},
\quad
\sup_{0<r\le1}r^a|\log r|=\frac1{ea},
\quad
\lim_{r\downarrow0}r^a|\log r|=0.
\end{equation}

For $s \neq 0$,
\[
\partial_s f(x, s) = |s|^{p(x)-2}\big[1 + (p(x)-1)\log|s|\big].
\]
For $0 < |h| \leq 1$, as $h\to 0$, 
\[
\left| \frac{f(x, h) - f(x, 0)}{h} \right|
= |h|^{p(x)-2}|\log|h||
\leq |h|^{p_- -2}|\log|h|| \to 0,
\]
so $\partial_s f(x, 0) = 0$. Moreover, for $0 < |s| \leq 1$, as $s\to 0$, 
\[
|\partial_s f(x, s)| \leq |s|^{p_- -2}\big[1 + (p_+ -1)|\log|s||\big] \to 0,
\]
hence $\partial_s f(x, \cdot)$ is continuous at $0$. Therefore $f(x, \cdot) \in C^1(\mathbb R)$.

(1) For $0<|s|\le1$, by using \eqref{eq22}, we obtain 
\begin{align*}
|f(x,s)|
&=|s|^{p(x)-1}|\log|s||
\le |s|^{p_--1}|\log|s||
\le\frac1{e(p_--1)},\\
|\partial_s f(x,s)|
&\le |s|^{p(x)-2}+(p(x)-1)|s|^{p(x)-2}|\log|s||\\
&\le1+(p_+-1)|s|^{p_--2}|\log|s||
\le1+\frac{p_+-1}{e(p_--2)},\\
|F(x,s)|
&\le\frac{|s|^{p(x)}}{p(x)}|\log|s||
+\frac{|s|^{p(x)}}{p(x)^2}\\
&\le\frac{|s|^{p_-}}{p_-}|\log|s||+\frac1{p_-^2}
\le\frac1{ep_-^2}+\frac1{p_-^2}.
\end{align*}
These bounds also hold at zero by continuity. 

(2) For $|s|\ge1$, using again \eqref{eq22}, we obtain 
\begin{align*}
|f(x,s)|
&=|s|^{p(x)-1}\log|s|
\le\frac1{e\sigma}|s|^{p(x)-1+\sigma}
\le\frac1{e\sigma}|s|^{p_+-1+\sigma},\\
|\partial_s f(x,s)|
&\le |s|^{p(x)-2}+(p_+-1)|s|^{p(x)-2}\log|s|\\
&\le |s|^{p_+-2}+\frac{p_+-1}{e\sigma}|s|^{p_+-2+\sigma}
\le\left(1+\frac{p_+-1}{e\sigma}\right)|s|^{p_+-2+\sigma},\\
|F(x,s)|
&\le\frac1{p_-}|s|^{p(x)}\log|s|+\frac1{p_-^2}|s|^{p(x)}\\
&\le\frac1{e\sigma p_-}|s|^{p_++\sigma}+\frac1{p_-^2}|s|^{p_+}
\le\left(\frac1{e\sigma p_-}+\frac1{p_-^2}\right)|s|^{p_++\sigma}.
\end{align*}
This completes the proof.
\end{proof}

\begin{lemma}[Lipschitz continuity of the truncated source]
\label{lem:cutoff}
Let $\vartheta\in C_c^\infty(\mathbb R)$ be even and satisfy
\[
0\le\vartheta\le1,
\quad
\vartheta(s)=1\quad\text{for }|s|\le1,
\quad
\vartheta(s)=0\quad\text{for }|s|\ge2.
\]
For $n\in\mathbb N$, set
\[
f_n(x,s)=\vartheta(s/n)f(x,s).
\]
For each fixed $n$, 
\begin{enumerate}
  \item the Nemytskii operator
$f_n:H\to H$ is globally Lipschitz continuous.
  \item the operators
\[
f_n:L^{q_\sigma}(\Omega)\to L^{m_-'}(\Omega), \quad f_n:L^{q_\sigma}(\Omega)\to L^{m'(x)}(\Omega)
\]
are locally Lipschitz continuous, with Lipschitz constants
on bounded sets independent of $n$. Here
\[q_\sigma:=(p_+-1+\sigma)m_-'.\]
\end{enumerate}
\end{lemma}
\begin{proof}

Clearly, 
\[
\partial_s f_n(x,s)
=
\vartheta(s/n)\partial_s f(x,s)
+\frac1n\vartheta'(s/n)f(x,s).
\]
Lemma~\ref{lem:growth} gives
\[
|\vartheta(s/n)\partial_s f(x,s)|
\le C(1+|s|^{p_+-2+\sigma}).\]
Since $\vartheta'(s/n)$ is supported on $n\le|s|\le2n$, where $|s|\ge1$ and
$1/n\le2/|s|$, by \eqref{eq22} we derive 
\[\left|
\frac1n\vartheta'(s/n)f(x,s)
\right|
\le
\frac{2\|\vartheta'\|_\infty}{|s|}
|s|^{p(x)-1}\log|s|
\le
\frac{2\|\vartheta'\|_\infty}{e\sigma}
|s|^{p(x)-2+\sigma}\le
\frac{2\|\vartheta'\|_\infty}{e\sigma}
|s|^{p_+-2+\sigma}.
\]
Combining the above two estimates arrives at 
\begin{equation}
\label{eq23}
|\partial_s f_n(x,s)|
\le C
\bigl(1+|s|^{p_+-2+\sigma}\bigr).
\end{equation}

(1) For fixed $n$, the derivative vanishes for $|s|>2n$.
Thus
\[
\sup_{s\in\mathbb R}|\partial_s f_n(x,s)|
\le
C\bigl(1+(2n)^{p_+-2+\sigma}\bigr)
=:C_n<\infty.
\]
Therefore the fundamental theorem of calculus  implies
\[
|f_n(x,a)-f_n(x,b)|\le C_n|a-b|.
\]
Since $f_n(x,0)=0$, we also have $|f_n(x,a)|\le C_n|a|$.
Hence $f_n$ maps $H$ into $H$.

(2)  To derive an estimate uniform in $n$, we apply \eqref{eq23}, then 
\begin{align*}
|f_n(x,a)-f_n(x,b)|
&=\left|(a-b)\int_0^1 \partial_s f_n(x,b+t(a-b))\,dt\right|\\
&\le C|a-b|\int_0^1 \big(1+|b+t(a-b)|^{p_+-2+\sigma}\big)\,dt\\
&\le C|a-b|\big(1+|a|^{p_+-2+\sigma}+|b|^{p_+-2+\sigma}\big),
\end{align*}
where we used $|b+t(a-b)|=|(1-t)b+ta|\le\max\{|a|,|b|\}$ for $0\le t\le1$ and $p_+-2+\sigma>0$.

By the definition of $q_\sigma$,
\[
\frac1{m_-'}
=
\frac1{q_\sigma}
+\frac{p_+-2+\sigma}{q_\sigma}.
\]
H\"older's inequality implies 
\begin{align*}
\bigl\|
|u-v|\,|u|^{p_+-2+\sigma}
\bigr\|_{m_-'}
\le
\|u-v\|_{q_\sigma}
\bigl\||u|^{p_+-2+\sigma}\bigr\|_
{q_\sigma/(p_+-2+\sigma)}=
\|u-v\|_{q_\sigma}
\|u\|_{q_\sigma}^{p_+-2+\sigma}.
\end{align*}
The same estimate holds with $u$ replaced by $v$.
Furthermore, since $q_\sigma>m_-'$ and $|\Omega|<\infty$,
\[
\|u-v\|_{m_-'}
\le
|\Omega|^{1/m_-'-1/q_\sigma}
\|u-v\|_{q_\sigma}.
\]
Combining these estimates yields
\begin{align*}
\|f_n(\cdot,u)-f_n(\cdot,v)\|_{m_-'}
\le
C
\left(
|\Omega|^{1/m_-'-1/q_\sigma}
+\|u\|_{q_\sigma}^{p_+-2+\sigma}
+\|v\|_{q_\sigma}^{p_+-2+\sigma}
\right)
\|u-v\|_{q_\sigma}.
\end{align*}
If $\|u\|_{q_\sigma}+\|v\|_{q_\sigma}\le L$, then
\begin{equation*}
\|f_n(\cdot,u)-f_n(\cdot,v)\|_{m_-'}
\le C
\left(
|\Omega|^{1/m_-'-1/q_\sigma}
+2L^{p_+-2+\sigma}
\right)\|u-v\|_{q_\sigma}.
\end{equation*}

The continuous inclusion
$L^{m_-'}(\Omega)\hookrightarrow L^{m'(x)}(\Omega)$  directly indicates that $f_n:L^{q_\sigma}(\Omega)\to L^{m'(x)}(\Omega)$ 
is locally Lipschitz continuous. 
\end{proof}

Next, we aim to prove Theorem~\ref{thm:local}. The argument has three stages: we formulate
the memory equation as an evolution governed by a maximal monotone
operator, construct solutions for the truncated sources, and obtain
uniform estimates that allow the truncation to be removed. The energy identity is then established for every energy solution.
Together with the local existence estimates and control of the
memory history, it yields the continuation criterion.

\subsection{The history formulation}
Extend $u$ by zero for negative times. Following Dafermos
\cite{Dafermos1970}, introduce the history variable
\begin{equation}
\label{1403}
\eta^t(s)=
\begin{cases}
u(t)-u(t-s),&0<s\le t,\\
u(t),&s>t.
\end{cases}
\end{equation}

For notational convenience, we suppress the superscript $t$
and write $\eta(s)$ for $\eta^t(s)$ whenever the time is clear
from the context. Then
\[
\eta_t+\eta_s=u_t,\quad
\eta^t(0)=0,\quad
\eta^0(s)=u_0\quad(s>0).
\]
The associated spaces and phase norm are defined as 
\begin{gather*}
\mathcal M=L_g^2(\mathbb R_+;V),\quad
\HH=V\times H\times\mathcal M,\\
\|(u,v,\eta)\|_{\HH}^2
=\ell\|\nabla u\|_2^2+\|v\|_2^2
+\int_0^\infty g(s)\|\nabla\eta(s)\|_2^2\,ds.
\end{gather*}
For the initial phase $U_0=(u_0,u_1,\eta^0)$,
\[
\|U_0\|_{\HH}^2
=\left(\ell+\int_0^\infty g(s)ds\right)\|\nabla u_0\|_2^2
+\|u_1\|_2^2=\|\nabla u_0\|_2^2+\|u_1\|_2^2.
\]

Set $A_0=-\Delta:V\to V^*$. Splitting the history integral at $s=t$ results in 
\begin{align*}
\ell A_0u(t)+\int_0^\infty g(s)A_0\eta^t(s)ds
&=\left(\ell+\int_0^\infty g(s)ds\right)A_0u(t)
-\int_0^t g(s)A_0u(t-s)ds\\
&=A_0u(t)-\int_0^t g(t-s)A_0u(s)ds.
\end{align*}
Thus the original equation is equivalent to 
\begin{equation*}
\begin{cases}
u_t=v,& \\
v_t +\ell A_0u+\int_0^\infty g(s)A_0\eta(s)ds+b(x,v)=f(x,u),&\\
\eta_t+\eta_s=v.&
\end{cases}
\end{equation*}

\subsection{The monotone operator}
Define
\[
\Op(u,v,\eta)=
\left(-v,\ \ell A_0u+\int_0^\infty g(s)A_0\eta(s)ds+b(x,v),\ \eta_s-v\right),
\]
with domain 
\[
\begin{aligned}
D(\Op)=\biggl\{(u,v,\eta):\;&u\in V,\quad v\in X_m,\quad
\eta,\eta_s\in\mathcal M,\quad\eta(0)=0,\\
&\ell A_0u+\int_0^\infty g(s)A_0\eta(s)\,ds+b(x,v)\in H\biggr\}.
\end{aligned}
\]
The domain $D(\Op)$ is dense in $\HH$.
For two elements belonging to $D(\Op)$, denote their difference by
$(\widehat u,\widehat v,\widehat\eta)$. By expanding the phase-space inner product, set $U:=(u,v,\eta)$, we arrive at
\begin{align*}
(\Op U_1-\Op U_2,U_1-U_2)_{\HH}
={}&-\ell(\nabla\widehat v,\nabla\widehat u)
+\ell(\nabla\widehat u,\nabla\widehat v)\\
&+\int_0^\infty g(s)(\nabla\widehat\eta,\nabla\widehat v)ds
-\int_0^\infty g(s)(\nabla\widehat v,\nabla\widehat\eta)ds\\
&+\int_\Omega\big[b(x,v_1)-b(x,v_2)\big]\widehat v\,dx
+\int_0^\infty g(s)(\nabla\widehat\eta_s,\nabla\widehat\eta)ds.
\end{align*}
Since $g(s)\ge g(R)>0$ on $[0,R]$, the domain conditions give
$\widehat\eta\in W^{1,2}(0,R;V)$. Hence
$g(s)\|\nabla\widehat\eta(s)\|_2^2$ is absolutely continuous on $[0,R]$.
Using $\widehat\eta(0)=0$, integration by parts gives
\[
\begin{aligned}
\int_0^R g(s)
\bigl(\nabla\widehat\eta_s(s),\nabla\widehat\eta(s)\bigr)\,ds
=\frac12g(R)\|\nabla\widehat\eta(R)\|_2^2+\frac12\int_0^R[-g'(s)]
\|\nabla\widehat\eta(s)\|_2^2\,ds.
\end{aligned}
\]
Because $\widehat\eta,\widehat\eta_s\in\mathcal M$, one has 
\[
\int_0^\infty g(s)\|\nabla\widehat\eta(s)\|_2^2\,ds
=\|\widehat\eta\|_{\mathcal M}^2<\infty,\quad
\int_0^\infty g(s)
\left|\bigl(\nabla\widehat\eta_s(s),
\nabla\widehat\eta(s)\bigr)\right|\,ds
\le
\|\widehat\eta_s\|_{\mathcal M}
\|\widehat\eta\|_{\mathcal M}<\infty.
\]
The first bound provides an increasing sequence $R_j\to\infty$
such that
$
g(R_j)\|\nabla\widehat\eta(R_j)\|_2^2\to0.
$
Passing to the limit along this sequence, using the second bound
and monotone convergence, yields
\begin{equation*}
\int_0^\infty g(s)
\bigl(\nabla\widehat\eta_s(s),\nabla\widehat\eta(s)\bigr)\,ds
=
\frac12\int_0^\infty[-g'(s)]
\|\nabla\widehat\eta(s)\|_2^2\,ds
\ge0.
\end{equation*}
Moreover, the monotonicity of $s\mapsto |s|^{m(x)-2}s$ implies
\[
\int_\Omega
\bigl[b(x,v_1)-b(x,v_2)\bigr](v_1-v_2)\,dx\ge0.
\]
Consequently,
\[
\begin{aligned}
(\Op U_1-\Op U_2,U_1-U_2)_{\HH}
&=\int_\Omega
\bigl[b(x,v_1)-b(x,v_2)\bigr](v_1-v_2)\,dx\\
&\quad+\frac12\int_0^\infty[-g'(s)]
\|\nabla\widehat\eta(s)\|_2^2\,ds
\ge0.
\end{aligned}
\]
Thus $\Op$ is monotone.

\subsection{The resolvent equation}
Fix $\lambda>0$ and $h=(h_1,h_2,h_3)\in\HH$.
Writing $U=(u,v,\eta)$, the resolvent equation
$U+\lambda\Op U=h$ takes the componentwise form
\begin{equation}
\label{eq24}
\begin{cases}
u-\lambda v=h_1,\\[1mm]
v+\lambda\ell A_0u
+\lambda\displaystyle\int_0^\infty g(s)A_0\eta(s)\,ds
+\lambda b(x,v)=h_2,\\[2mm]
\eta(s)+\lambda\eta_s(s)-\lambda v=h_3(s),
\quad s>0,\\[1mm]
\eta(0)=0.
\end{cases}
\end{equation}
It follows from $\eqref{eq24}_1$, $\eqref{eq24}_3$ with $\eta(0)=0$ that 
\begin{equation}
\label{eq25}
u=h_1+\lambda v,\quad
\eta(s)=\lambda(1-e^{-s/\lambda})v
+\frac1\lambda\int_0^s
e^{-(s-r)/\lambda}h_3(r)\,dr.
\end{equation}
Let us define
\[
(R_\lambda h_3)(s)
:=\frac1\lambda\int_0^s
e^{-(s-r)/\lambda}h_3(r)\,dr,
\quad s>0.
\]
We are now in a position to show that $R_\lambda h_3\in\mathcal M$ and
$
\|R_\lambda h_3\|_{\mathcal M}
\le \|h_3\|_{\mathcal M}.
$
First, since $g$ is positive and nonincreasing, for every $s>0$,
\begin{equation}
\label{eq26}
\int_0^s\|\nabla h_3(r)\|_2^2\,dr
\le \frac1{g(s)}
\int_0^s g(r)\|\nabla h_3(r)\|_2^2\,dr
\le \frac1{g(s)}\|h_3\|_{\mathcal M}^2.
\end{equation}
Consequently, $h_3\in L^2(0,s;V)$ for every finite $s>0$,
so $R_\lambda h_3$ is a well-defined Bochner integral in $V$. By
the Cauchy--Schwarz inequality, we obtain 
\begin{align*}
\|\nabla(R_\lambda h_3)(s)\|_2^2
&\le
\left(
\int_0^s\frac1\lambda e^{-(s-r)/\lambda}
\|\nabla h_3(r)\|_2\,dr
\right)^2\\
&\le
\left(
\int_0^s\frac1\lambda e^{-(s-r)/\lambda}\,dr
\right)
\left(
\int_0^s\frac1\lambda e^{-(s-r)/\lambda}
\|\nabla h_3(r)\|_2^2\,dr
\right)\\
&\le
\int_0^s\frac1\lambda e^{-(s-r)/\lambda}
\|\nabla h_3(r)\|_2^2\,dr.
\end{align*}
Multiplying the above inequality by $g(s)$ and integrating over $(0,\infty)$,
then applying Tonelli's theorem and using
$g(s)\le g(r)$ for $s\ge r$, we obtain 
\begin{equation}
\label{eq27}
\begin{split}
\|R_\lambda h_3\|_{\mathcal M}^2
&\le\int_0^\infty\int_0^s
g(s)\frac1\lambda e^{-(s-r)/\lambda}
\|\nabla h_3(r)\|_2^2\,dr\,ds\\
&=\int_0^\infty
\|\nabla h_3(r)\|_2^2
\left(
\int_r^\infty
g(s)\frac1\lambda e^{-(s-r)/\lambda}\,ds
\right)dr\\
&\le\int_0^\infty
g(r)\|\nabla h_3(r)\|_2^2
\left(
\int_r^\infty
\frac1\lambda e^{-(s-r)/\lambda}\,ds
\right)dr\\
&=\int_0^\infty
g(r)\|\nabla h_3(r)\|_2^2\,dr=\|h_3\|_{\mathcal M}^2.
\end{split}
\end{equation}
Thus $R_\lambda$ is a bounded linear operator on
$\mathcal M$.

Substituting \eqref{eq25} into $\eqref{eq24}_2$, we obtain
\begin{align*}
v+\lambda b(x,v)
+\lambda^2
\left[
\ell+\int_0^\infty g(s)(1-e^{-s/\lambda})\,ds
\right]A_0v=h_2-\lambda\ell A_0h_1
-\lambda\int_0^\infty g(s)A_0R_\lambda h_3(s)\,ds.
\end{align*}
Equivalently, we seek $v\in X_m$ such that, for every
$z\in X_m$,
\begin{equation}
\label{eq28}
\begin{split}
&\int_\Omega vz\,dx+\lambda\int_\Omega b(x,v)z\,dx+\lambda^2
\left[\ell+\int_0^\infty g(s)(1-e^{-s/\lambda})\,ds\right]
\int_\Omega\nabla v\cdot\nabla z\,dx\\
&\quad=\int_\Omega h_2z\,dx
-\lambda\ell\int_\Omega\nabla h_1\cdot\nabla z\,dx-\lambda\int_0^\infty g(s)
\int_\Omega\nabla(R_\lambda h_3)(s)\cdot\nabla z\,dx\,ds.
\end{split}
\end{equation}
Let us first verify that the right-hand side defines a bounded
linear functional on $X_m$. By Cauchy--Schwarz and Poincar\'e's inequality,
\begin{align*}
\left|\int_\Omega h_2z\,dx\right|
\le\|h_2\|_2\|z\|_2
\le C_p\|h_2\|_2\|\nabla z\|_2,\quad
\lambda\ell\left|\int_\Omega\nabla h_1\cdot\nabla z\,dx\right|\le\lambda\ell\|\nabla h_1\|_2\|\nabla z\|_2.
\end{align*}
For the history term, the Cauchy--Schwarz inequality
and \eqref{eq27} give 
\begin{align*}
\lambda\left|
\int_0^\infty g(s)
\int_\Omega
\nabla R_\lambda h_3(s)\cdot\nabla z\,dx\,ds
\right|
&\le
\lambda\|\nabla z\|_2
\left(\int_0^\infty g(s)\,ds\right)^{1/2}
\|R_\lambda h_3\|_{\mathcal M}\\
&\le
\lambda(1-\ell)^{1/2}
\|h_3\|_{\mathcal M}\|\nabla z\|_2.
\end{align*}
Thus, the right-hand side of \eqref{eq28} is bounded in absolute value by
\[
\left[
C_p\|h_2\|_2
+\lambda\ell\|\nabla h_1\|_2
+\lambda(1-\ell)^{1/2}\|h_3\|_{\mathcal M}
\right]\|z\|_{X_m}.
\]

Define $T_\lambda:X_m\to X_m^*$ by
\begin{align*}
\langle T_\lambda v,z\rangle
:=\int_\Omega vz\,dx
+\lambda\int_\Omega b(x,v)z\,dx+\lambda^2
\left[
\ell+\int_0^\infty g(s)(1-e^{-s/\lambda})\,ds
\right]
\int_\Omega\nabla v\cdot\nabla z\,dx.
\end{align*}
\textup{(1) Boundedness on bounded sets.}
For $v\in X_m$, the Luxemburg norm estimate gives
\[
\|b(x,v)\|_{m'(x)}
\le
\max\left\{
\|v\|_{m(x)}^{m_--1},
\|v\|_{m(x)}^{m_+-1}
\right\}.
\]
Indeed, for $a=\|v\|_{m(x)}>0$, take
$M=\max\{a^{m_--1},a^{m_+-1}\}$. Then
$M\ge a^{m(x)-1}$ and
\[
\int_\Omega\left|\frac{b(x,v)}{M}\right|^{m'(x)}\,dx
\le\int_\Omega\left|\frac{v}{a}\right|^{m(x)}\,dx\le1.
\]
The case $v=0$ is immediate.

By Poincar\'e's inequality, H\"older's inequality \eqref{eq3},
and $0<\ell+\int_0^\infty g(s)(1-e^{-s/\lambda})\,ds\le1$,
\[
\begin{aligned}
|\langle T_\lambda v,z\rangle|
&\le\left[(C_p^2+\lambda^2)\|\nabla v\|_2
+2\lambda\|b(x,v)\|_{m'(x)}\right]\|z\|_{X_m}\\
&\le\left[(C_p^2+\lambda^2)\|v\|_{X_m}
+2\lambda\max\left\{
\|v\|_{X_m}^{m_--1},\|v\|_{X_m}^{m_+-1}
\right\}\right]\|z\|_{X_m}.
\end{aligned}
\]
Taking the supremum over $\|z\|_{X_m}\le1$ shows that
$T_\lambda$ maps bounded subsets of $X_m$ into bounded
subsets of $X_m^*$.

\textup{(2) Hemicontinuity.}
Fix $v,w,z\in X_m$ and let $t_j\to t$ in $\mathbb R$.
The linear terms in $\langle T_\lambda(v+t_jw),z\rangle$
converge to their values at $t$. For the nonlinear term,
continuity of $s\mapsto |s|^{m(x)-2}s$ results in
\[
b(x,v+t_jw)z\to b(x,v+tw)z
\quad\text{for a.e. }x\in\Omega.
\]
Since $(t_j)$ is bounded, choose $k>0$ such that
$|t_j|\le k$ for all $j$. Young's inequality \eqref{eq4}
and $(a+b)^{m(x)}\le2^{m(x)-1}(a^{m(x)}+b^{m(x)})$ imply
\[
|b(x,v+t_jw)z|
\le (|v|+k|w|)^{m(x)-1}|z|
\le C\bigl(|v|^{m(x)}+|w|^{m(x)}+|z|^{m(x)}\bigr).
\]
The right-hand side is integrable on $\Omega$. Thus the
dominated convergence theorem shows convergence of the
nonlinear term, and hence
\[
\langle T_\lambda(v+t_jw),z\rangle
\to
\langle T_\lambda(v+tw),z\rangle.
\]
Therefore, $T_\lambda$ is hemicontinuous.

\textup{(3) Strict monotonicity.}
For $v_1,v_2\in X_m$ with $v_1\ne v_2$, we have
\begin{align*}
\langle T_\lambda v_1-T_\lambda v_2,v_1-v_2\rangle&=\|v_1-v_2\|_2^2+\lambda\int_\Omega
\bigl[b(x,v_1)-b(x,v_2)\bigr](v_1-v_2)\,dx\\
&\quad+\lambda^2
\left[
\ell+\int_0^\infty g(s)(1-e^{-s/\lambda})\,ds
\right]
\|\nabla(v_1-v_2)\|_2^2>0.
\end{align*}
Here we used that $s\mapsto |s|^{m(x)-2}s$ is increasing for almost every $x$. Thus $T_\lambda$ is strictly monotone.

\textup{(4) Coercivity.}
By the definition of $T_\lambda$,
\[
\langle T_\lambda v,v\rangle
\ge
\lambda\int_\Omega |v|^{m(x)}\,dx
+\lambda^2\ell\|\nabla v\|_2^2.
\]
We consider two cases.

If $\|\nabla v\|_2\ge\|v\|_{m(x)}$, then
$\|v\|_{X_m}\le2\|\nabla v\|_2$, and hence
\[
\frac{\langle T_\lambda v,v\rangle}{\|v\|_{X_m}}
\ge\frac{\lambda^2\ell}{2}\|\nabla v\|_2
\ge\frac{\lambda^2\ell}{4}\|v\|_{X_m}.
\]

If $\|v\|_{m(x)}>\|\nabla v\|_2$, then
$\|v\|_{X_m}<2\|v\|_{m(x)}$. In particular,
$\|v\|_{m(x)}>1$ whenever $\|v\|_{X_m}>2$. By \eqref{eq2},
\[
\frac{\langle T_\lambda v,v\rangle}{\|v\|_{X_m}}
\ge
\frac{\lambda}{2\|v\|_{m(x)}}
\int_\Omega |v|^{m(x)}\,dx
\ge
\frac{\lambda}{2}\|v\|_{m(x)}^{m_--1}
\ge
\frac{\lambda}{2^{m_-}}\|v\|_{X_m}^{m_--1}.
\]
Thus, in either case,
\[
\frac{\langle T_\lambda v,v\rangle}{\|v\|_{X_m}}
\to+\infty
\quad\text{as }\|v\|_{X_m}\to\infty,
\]
so $T_\lambda$ is coercive.

Since $X_m$ is reflexive and $T_\lambda:X_m\to X_m^*$
is bounded on bounded sets, hemicontinuous, monotone,
and coercive, the Browder--Minty surjectivity theorem \cite{Barbu} implies that
$T_\lambda$ is surjective.  Therefore \eqref{eq28} has a solution
$v\in X_m$. Strict monotonicity also makes this solution
unique.

It remains to verify that the reconstructed triple belongs
to $D(\Op)$. Recall \eqref{eq25} together with $R_\lambda h_3\in\mathcal M$, clearly $u\in V$, $v\in X_m$, and  $\eta\in\mathcal M$.

For every finite $S>0$, recall \eqref{eq26}, we have 
$R_\lambda h_3$ belongs to $W^{1,2}(0,S;V)$
and satisfies
\[
\frac{d}{ds}R_\lambda h_3(s)
=\lambda^{-1}h_3(s)-\lambda^{-1}R_\lambda h_3(s)
\]
for almost every $s\in(0,S)$. Therefore, from \eqref{eq25}, we have 
\begin{align*}
\eta_s(s)
=e^{-s/\lambda}v
+\lambda^{-1}h_3(s)-\lambda^{-1}R_\lambda h_3(s)=\lambda^{-1}\bigl(h_3(s)-\eta(s)\bigr)+v.
\end{align*}
Since $h_3,\eta\in\mathcal M$ and
\[
\int_0^\infty g(s)\|\nabla v\|_2^2\,ds
=(1-\ell)\|\nabla v\|_2^2<\infty.
\]
it follows that $\eta_s\in\mathcal M$.

As $s\to 0$, we have
\begin{align*}
\|\lambda(1-e^{-s/\lambda})v\|_V+\|R_\lambda h_3(s)\|_V
&\le \lambda(1-e^{-s/\lambda})\|v\|_V+
\lambda^{-1}\int_0^s\|h_3(r)\|_V\,dr\\
&\le\lambda(1-e^{-s/\lambda})\|v\|_V+\frac{\sqrt{s}}{\lambda}
\left(\int_0^s\|h_3(r)\|_V^2\,dr\right)^{1/2}
\to 0.
\end{align*}
Thus $\eta(0)=0$ in $V$.

Finally, \eqref{eq28} gives
\[
\ell A_0u+\int_0^\infty g(s)A_0\eta(s)\,ds+b(x,v)
=\frac{h_2-v}{\lambda}\in H.
\]
We have therefore verified all the
conditions defining $D(\Op)$, and the reconstructed triple
satisfies $U+\lambda\Op U=h$. Since $h\in\HH$ was arbitrary,
\[
\operatorname{Range}(I+\lambda\Op)=\HH,\quad \lambda>0.
\]
Together with the monotonicity established above, one shows that $\Op$ is maximal monotone on $\HH$.

\subsection{Truncated solutions and their energy identities}

For each fixed $n$, Lemma~\ref{lem:cutoff}(1) and Poincar\'e's
inequality give, for all $U,\widetilde U\in\HH$,
\[
\begin{aligned}
\|(0,f_n(x,u)-f_n(x,\widetilde u),0)\|_{\HH}
\le C_n\|u-\widetilde u\|_2
\le C_nC_p\|\nabla(u-\widetilde u)\|_2
\le\frac{C_nC_p}{\sqrt{\ell}}
\|U-\widetilde U\|_{\HH}.
\end{aligned}
\]
Thus $U\mapsto(0,f_n(x,u),0)$ is globally Lipschitz on $\HH$.
Since $\Op$ is maximal monotone and
$\overline{D(\Op)}^{\,\HH}=\HH$, the Lipschitz perturbation
theorem \cite{Brezis,Zheng2004}
yields, for every $U_0\in\HH$, a unique global \emph{mild} solution
$U\in C([0,\infty);\HH)$ of
\[
U_t+\Op U=(0,f_n(x,u),0),
\quad U(0)=U_0.
\]

Choose a sequence $U_{0,j}\in D(\Op)$ such that
$U_{0,j}\to U_0$ in $\HH$, and let
$U_j=(u_j,v_j,\eta_j)$ be the corresponding strong
solutions of
\[
U_j'+\Op U_j=(0,f_n(x,u_j),0),
\quad U_j(0)=U_{0,j}.
\]
For any $j,k$, subtracting their equations and pairing
with $U_j-U_k$ yields
\begin{equation}\label{eq29}
\begin{aligned}
&\frac12\frac d{dt}\|U_j-U_k\|_{\HH}^2
+\int_\Omega
[b(x,v_j)-b(x,v_k)](v_j-v_k)\,dx\\
&\quad+\frac12\int_0^\infty[-g'(s)]
\|\nabla(\eta_j-\eta_k)(s)\|_2^2\,ds\\
&=(f_n(x,u_j)-f_n(x,u_k),v_j-v_k)\\
&\le C_n\|u_j-u_k\|_2\|v_j-v_k\|_2
\le C_n\|U_j-U_k\|_{\HH}^2
\end{aligned}
\end{equation}
for almost every $t>0$.
Discarding the nonnegative dissipation terms and
applying Gronwall's inequality, we obtain, for every
$T>0$,
\[
\sup_{0\le t\le T}
\|U_j(t)-U_k(t)\|_{\HH}^2
\le e^{2C_nT}\|U_{0,j}-U_{0,k}\|_{\HH}^2.
\]
Since $\{U_{0,j}\}$ converges in $\HH$, the sequence
$\{U_j\}$ is Cauchy in $C([0,T];\HH)$. Therefore, there
exists $U\in C([0,T];\HH)$ such that as \(j\to\infty\)
\begin{equation}
\label{1150}
U_j\to U\quad\text{strongly in }C([0,T];\HH).
\end{equation}

Using
\begin{equation}
\label{eq30}
[b(x,a)-b(x,c)](a-c)
\ge 2^{2-m_+}|a-c|^{m(x)},
\end{equation}
integrating  \eqref{eq29} from $0$ to $T$ with respect to $t$ and using the preceding
$C([0,T];\HH)$ estimate, we obtain
\begin{equation}
\label{1147}
\begin{split}
&2^{2-m_+}\int_0^T\int_\Omega
|v_j-v_k|^{m(x)}\,dx\,dt+\frac12\int_0^T\int_0^\infty[-g'(s)]
\|\nabla(\eta_j-\eta_k)(t,s)\|_2^2\,ds\,dt\\
&\le
\left(\frac12+C_nTe^{2C_nT}\right)
\|U_{0,j}-U_{0,k}\|_{\HH}^2
\to0 \quad\text{as }j,k\to\infty.
\end{split}
\end{equation}
Thus, as \(j\to\infty\), 
\begin{equation}
\label{1200}
v_j\to v\quad\text{strongly in }L^{m(x)}(Q_T).
\end{equation}
The dissipation estimate \eqref{1147} shows that $\nabla\eta_j$ is Cauchy in
$L^2((0,T)\times(0,\infty)\times\Omega;
[-g'(s)]\,dt\,ds\,dx)$. Denote its limit in this space by $Z$.
On the other hand, \eqref{1150} gives
\[
\int_0^T\int_0^\infty g(s)
\|\nabla(\eta_j-\eta)(t,s)\|_2^2\,ds\,dt\to0.
\]
A subsequence therefore converges to $\nabla\eta$ almost everywhere
with respect to $g(s)\,dt\,ds\,dx$, and hence with respect to Lebesgue
measure, since $g>0$. A further subsequence converges to $Z$ almost
everywhere with respect to $[-g'(s)]\,dt\,ds\,dx$. This latter measure
is absolutely continuous with respect to Lebesgue measure. Consequently,
$Z=\nabla\eta$ almost everywhere, and
convergence of the original Cauchy sequence yields
\[
\int_0^T\int_0^\infty[-g'(s)]
\|\nabla(\eta_j-\eta)(t,s)\|_2^2\,ds\,dt
\to0 \quad\text{as }j\to\infty.
\]
Moreover, Lemma~\ref{lem:cutoff}(1), \(f_n(0)=0\) and \eqref{1150} produce 
\begin{align*}
&\left|
\int_0^t
\bigl[(f_n(x,u_j),v_j)-(f_n(x,u),v)\bigr]\,d\tau
\right|\\
&\quad\le
C_nT\sup_{[0,T]}\|u_j-u\|_2
\sup_{[0,T]}\|v_j\|_2+
T\sup_{[0,T]}\|f_n(x,u)\|_2
\sup_{[0,T]}\|v_j-v\|_2
\to0 \quad\text{as }j\to\infty.
\end{align*}
The preceding strong convergences and the global Lipschitz
continuity of $f_n:H\to H$ allow passage to the limit in
the strong energy identity, yielding
 the energy identity
\begin{equation}
\label{eq31}
\begin{split}
&\frac12\|U(t)\|_{\HH}^2
+\int_0^t\int_\Omega |v|^{m(x)}\,dx\,d\tau+\frac12\int_0^t\int_0^\infty[-g'(s)]
\|\nabla\eta(\tau,s)\|_2^2\,ds\,d\tau\\
&\quad=\frac12\|U_0\|_{\HH}^2
+\int_0^t(f_n(x,u),v)\,d\tau.
\end{split}
\end{equation}

Let us pause here to verify the continuity of
$b:L^{m(x)}(Q_T)\to L^{m'(x)}(Q_T)$.
Recall $v_j\to v$ in $L^{m(x)}(Q_T)$ in \eqref{1200}, then
$v_j\to v$ in measure and
\[
|b(x,v_j)-b(x,v)|^{m'(x)}
\le C\bigl(|v_j-v|^{m(x)}+|v|^{m(x)}\bigr).
\]
The right-hand side implies uniform integrability. Since
$b(x,v_j)\to b(x,v)$ in measure, Vitali's theorem and
\eqref{eq2} yield
\[
b(x,v_j)\to b(x,v)
\quad\text{strongly in }L^{m'(x)}(Q_T)\quad\text{as }j\to\infty.\]

Fix $n,k\in\mathbb N$, and let $U_n$ and $U_k$ be the
mild solutions corresponding to the truncated sources
$f_n$ and $f_k$, with initial data
$U_{0,n},U_{0,k}\in\HH$, respectively. Choose
$
U_{0,n,j},U_{0,k,j}\in D(\Op),$ $
U_{0,n,j}\to U_{0,n},~
U_{0,k,j}\to U_{0,k} \text{ in }\HH$ as \(j\to\infty\). Let $U_{n,j}$ and $U_{k,j}$ be the corresponding
strong solutions. Applying the same approximation argument to two
fixed truncated problems, and using the continuity
of $b:L^{m(x)}(Q_T)\to L^{m'(x)}(Q_T)$, we obtain
the following difference energy identity for their
mild solutions\begin{equation}
\label{eq32}
\begin{split}
&\frac12\|U_n(t)-U_k(t)\|_{\HH}^2+\int_0^t\int_\Omega
[b(x,v_n)-b(x,v_k)](v_n-v_k)\,dx\,d\tau\\
&\qquad+\frac12\int_0^t\int_0^\infty[-g'(s)]
\|\nabla(\eta_n-\eta_k)(\tau,s)\|_2^2\,ds\,d\tau\\
&\quad=\frac12\|U_n(0)-U_k(0)\|_{\HH}^2+\int_0^t
(f_n(x,u_n)-f_k(x,u_k),v_n-v_k)\,d\tau.
\end{split}
\end{equation}
Here only $\eta^0\in\mathcal M$ is required. 
In particular, if $U_{0,n}=U_{0,k}=U_0$, the initial
difference term vanishes.

\subsection{A common existence interval}
Let $U_n=(u_n,v_n,\eta_n)$ be the solution of
\[
U_n'+\Op U_n=(0,f_n(x,u_n),0),\quad U_n(0)=U_0.
\]
The energy identity \eqref{eq31}, Young's inequality \eqref{eq4}, and
Lemma~\ref{lem:growth} yield
\begin{equation}
\label{eq33}
\begin{split}
&\frac12\frac d{dt}\|U_n\|_{\HH}^2
+\int_\Omega|v_n|^{m(x)}dx
+\frac12\int_0^\infty[-g'(s)]\|\nabla\eta_n(s)\|_2^2ds=\int_\Omega f_n(x,u_n)v_n\,dx\\
&\qquad\le\frac12\int_\Omega|v_n|^{m(x)}dx
+C\int_\Omega|f_n(x,u_n)|^{m'(x)}dx\\
&\qquad\le\frac12\int_\Omega|v_n|^{m(x)}dx
+C\int_\Omega(1+|u_n|^{q_\sigma})dx.
\end{split}
\end{equation}
Here we uesd
\[
|f_n(x,s)|^{m'(x)}
\le C\bigl(1+|s|^{q_\sigma}\bigr).\]
Indeed, since $|f_n(x,s)|\le |f(x,s)|$, Lemma~\ref{lem:growth}
gives a uniform bound for $|f_n(x,s)|$ when $|s|\le1$.
For $|s|>1$, it yields
\[
|f_n(x,s)|^{m'(x)}
\le C|s|^{(p_+-1+\sigma)m'(x)}
\le C|s|^{(p_+-1+\sigma)m_-'}
=C|s|^{q_\sigma},
\]
where we used $m'(x)\le m_-'$.

 Since $q_\sigma<2^*$, Sobolev embedding $V\hookrightarrow L^{q_\sigma}(\Omega)$ and the definition
of the phase norm show
\[
\|u_n\|_{q_\sigma}^{q_\sigma}
\le C\|\nabla u_n\|_2^{q_\sigma}
\le C\|U_n\|_{\HH}^{q_\sigma}.
\]
Then \eqref{eq33} produces
\[
\frac d{dt}\left(1+\tfrac12\|U_n\|_{\HH}^2\right)
\le C\left(1+\tfrac12\|U_n\|_{\HH}^2\right)^{\max\{1,q_\sigma/2\}}.
\]
Set $Z_0=1+\frac12\|U_0\|_{\HH}^2$ and
$a=\max\{1,q_\sigma/2\}$. Comparison with
$z'=Cz^a$, $z(0)=Z_0$, results in 
\[
1+\frac12\|U_n(t)\|_{\HH}^2\le z(t),
\quad
z(t)=
\begin{cases}
\bigl[Z_0^{1-a}-C(a-1)t\bigr]^{-1/(a-1)},&a>1,\\
Z_0e^{Ct},&a=1.
\end{cases}
\]
Choose
\[
T_0=
\begin{cases}
\displaystyle
\frac{(1-2^{1-a})Z_0^{1-a}}{C(a-1)},&a>1,\\[6pt]
\displaystyle\frac{\log2}{C},&a=1.
\end{cases}
\]
Then $T_0>0$ is independent of $n$, and
\begin{equation}
\label{eq34}
\sup_{0\le t\le T_0}
\left(1+\frac12\|U_n(t)\|_{\HH}^2\right)\le2Z_0.
\end{equation}
Integrating \eqref{eq33} over $[0,T_0]$ yields
\[
\begin{aligned}
\int_0^{T_0}\int_\Omega |v_n|^{m(x)}\,dx\,dt+\int_0^{T_0}\int_0^\infty[-g'(s)]
\|\nabla\eta_n(t,s)\|_2^2\,ds\,dt\le \|U_0\|_{\HH}^2+2CT_0(2Z_0)^a.
\end{aligned}
\]

\subsection{Passage to the original source} 
It follows from \eqref{eq34} that  
\[\sup_n\|u_n\|_{L^\infty(0,T_0;V)}<\infty,
\quad
\sup_n\|v_{n}\|_{L^\infty(0,T_0;H)}<\infty.\]
Therefore,
\[\|u_n(t)-u_n(s)\|_2
\le\int_s^t\|v_n(\tau)\|_2\,d\tau
\le C|t-s|.\]
For $2<a<2^*$ choose $\theta\in(0,1)$ with
$1/a=\theta/2+(1-\theta)/2^*$. Interpolation inequality and Sobolev embedding $V\hookrightarrow L^{2^*}(\Omega)$ deliver
\[
\|u_n(t)-u_n(s)\|_a
\le\|u_n(t)-u_n(s)\|_2^\theta
\|u_n(t)-u_n(s)\|_{2^*}^{1-\theta}\le C|t-s|^\theta.
\]
For $1\le a\le2$, H\"{o}lder's inequality shows 
\[\|u_n(t)-u_n(s)\|_a
\le|\Omega|^{1/a-1/2}\|u_n(t)-u_n(s)\|_2
\le C|t-s|.\]
For each $t\in[0,T_0]$, the uniform bound in $V$ and the
compact embedding $V\hookrightarrow\hookrightarrow L^a(\Omega)$
imply that $\{u_n(t)\}$ is relatively compact in $L^a(\Omega)$
for every $1\le a<2^*$. Together with the time equicontinuity
established above, the Arzelà--Ascoli theorem and a diagonal
argument yield a subsequence, still denoted by $\{u_n\}$, such that  as $n\to\infty$
\[
u_n\to u
\quad\text{strongly in }C([0,T_0];L^a(\Omega)),
\quad 1\le a<2^*.
\]
In particular this holds for $a=q_\sigma$ and $a=p_++\sigma$.

Decompose
\[
f_n(x,u_n)-f(x,u)
=[f_n(x,u_n)-f_n(x,u)]+[f_n(x,u)-f(x,u)].
\]
Since $u_n\to u$ in $C([0,T_0];L^{q_\sigma}(\Omega))$ as $n\to\infty$, it follows from Lemma
\ref{lem:cutoff}(2) that 
\[
\begin{aligned}
\|f_n(x,u_n)-f_n(x,u)\|_{L^{m_-'}(Q_{T_0})}
&\quad=
\left(\int_0^{T_0}
\|f_n(x,u_n(t))-f_n(x,u(t))\|_{m_-'}^{m_-'}\,dt
\right)^{1/m_-'}\\
&\quad\le CT_0^{1/m_-'}
\sup_{0\le t\le T_0}\|u_n(t)-u(t)\|_{q_\sigma}
\to 0\quad\text{as }n\to\infty.
\end{aligned}
\]

By $m'(x)\le m_-'$ and $|Q_{T_0}|<\infty$, this convergence
also holds in $L^{m'(x)}(Q_{T_0})$.
Moreover, $f_n(x,u)\to f(x,u)$ almost everywhere as $n\to\infty$, and
\[
|f_n(x,u)-f(x,u)|^{m'(x)}=|\vartheta(u/n)-1|\,|f(x,u)|^{m'(x)}
\le |f(x,u)|^{m'(x)}
\le C(1+|u|^{q_\sigma})\in L^1(Q_{T_0}).
\]
Here we used Lemma \ref{lem:growth}. 
By dominated convergence theorem, as $n\to\infty$, one obtains 
\begin{equation}\label{eq35}
f_n(x,u_n)\to f(x,u)
\quad\text{strongly in }L^{m'(x)}(Q_{T_0}).
\end{equation}
Recalling \eqref{eq32} and using \eqref{eq3}, \eqref{eq35}, and \eqref{eq30}, as $n\to\infty$, we obtain 
\begin{equation*}
\begin{split}
&\frac12\|U_n(t)-U_k(t)\|_{\HH}^2+2^{2-m_+}\int_0^t\int_\Omega|v_n-v_k|^{m(x)}\,dx\,d\tau\\
&\qquad+\frac12\int_0^t\int_0^\infty[-g'(s)]
\|\nabla(\eta_n-\eta_k)(\tau,s)\|_2^2\,ds\,d\tau\\
&\quad \le \frac12\|U_n(0)-U_k(0)\|_{\HH}^2+2\|f_n(x,u_n)-f_k(x,u_k)\|_{L^{m'(x)}(Q_{T_0})}
\|v_n-v_k\|_{L^{m(x)}(Q_{T_0})}\to0.
\end{split}
\end{equation*}
Therefore, we obtain, along the extracted subsequence, as $n\to\infty$, 
\[
U_n\to U\quad\text{strongly in }C([0,T_0];\HH),
\quad
v_n\to v\quad\text{strongly in }L^{m(x)}(Q_{T_0}),
\]
and
\[
\int_0^{T_0}\int_0^\infty[-g'(s)]
\|\nabla(\eta_n-\eta)(t,s)\|_2^2\,ds\,dt
\to0.
\]
Since $v_n\to v$ strongly in $L^{m(x)}(Q_{T_0})$ as $n\to\infty$, one may extract a subsequence satisfying
\[
v_n\to v\quad\text{a.e. in }Q_{T_0}.
\]
The continuity of $b(x,\cdot)$
therefore gives $b(x,v_n)\to b(x,v)$ a.e. in $Q_{T_0}$. Since $v_n\to v$ strongly in $L^{m(x)}(Q_{T_0})$,
the continuity of
$b:L^{m(x)}(Q_{T_0})\to L^{m'(x)}(Q_{T_0})$
implies that 
\[b(x,v_n)\to b(x,v)\text{ strongly in }
L^{m'(x)}(Q_{T_0}).\]


We now pass to the limit in the weak formulation
of the truncated equation. 
For every $z\in X_m$ and $\varphi\in C_c^\infty(0,T_0)$,
the convergence $v_n\to v$ strongly in $C([0,T_0];H)$ gives
\[
-\int_0^{T_0}(v_n,z)\varphi'(t)\,dt
\to
-\int_0^{T_0}(v,z)\varphi'(t)\,dt.
\]
The elastic and memory terms converge by the strong
convergence in $C([0,T_0];\HH)$, while the source and
damping terms converge strongly in $L^{m'(x)}(Q_{T_0})$.
Moreover, passing to the limit in
\[
u_n(t)=u_0+\int_0^t v_n(\tau)\,d\tau
\]
yields $u(t)=u_0+\int_0^t v(\tau)\,d\tau$ in $H$, so
$u_t=v$. Consequently, for every $z\in X_m$, \eqref{eq6} holds
in $\mathcal D'(0,T_0)$.
The initial conditions follow from the convergence in
$C([0,T_0];\HH)$. Thus $u$ is an energy solution on
$[0,T_0]$.

\subsection{The energy identity}
We prove \eqref{eq7} for an arbitrary energy
solution in Definition~\ref{def:energy}.
By Lemma~\ref{lem:growth} and $u\in C([0,T];V)$,
\[
f(x,u)\in L^{m'(x)}(Q_T).
\]
Moreover,
\[
\int_{Q_T}|b(x,u_t)|^{m'(x)}\,dx\,dt
=\int_{Q_T}|u_t|^{m(x)}\,dx\,dt<\infty.
\]

Fix $0<s<t<T$. For $0<h<\min\{s,T-t\}$, define
\[
v_h(\tau)
=\frac{u(\tau+h)-u(\tau-h)}{2h}
=\frac1{2h}\int_{\tau-h}^{\tau+h}u_t(r)\,dr,
\quad \tau\in[s,t].
\]
For each fixed $h$, the regularity of $u$ and $u_t$
and Jensen's inequality imply
\[
v_h\in C([s,t];V)
\cap L^{m(x)}(\Omega\times(s,t)),
\quad
\partial_\tau v_h
=\frac{u_t(\tau+h)-u_t(\tau-h)}{2h}
\in C([s,t];H).
\]
The uniform continuity of $u_t$ with values in $H$ shows 
\[
\begin{aligned}
\sup_{\tau\in[s,t]}\|v_h(\tau)-u_t(\tau)\|_2
&\le
\sup_{\tau\in[s,t]}
\frac1{2h}\int_{-h}^{h}
\|u_t(\tau+r)-u_t(\tau)\|_2\,dr\\
&\le
\sup_{\substack{\tau\in[s,t]\\|r|\le h}}
\|u_t(\tau+r)-u_t(\tau)\|_2
\to0\quad\text{as }h\to0.
\end{aligned}
\]
Moreover, Jensen's inequality and Fubini's theorem yield
\[
\begin{aligned}
&\int_s^t\int_\Omega
|v_h(x,\tau)-u_t(x,\tau)|^{m(x)}\,dx\,d\tau\\
&\quad\le
\frac1{2h}\int_{-h}^{h}
\int_s^t\int_\Omega
|u_t(x,\tau+r)-u_t(x,\tau)|^{m(x)}
\,dx\,d\tau\,dr
\to0\quad\text{as }h\to0.
\end{aligned}
\]
Thus, as \(h\to0\), 
\[
v_h\to u_t
\quad\text{strongly in }C([s,t];H)\cap L^{m(x)}(\Omega\times(s,t)).
\]
We first regularize \eqref{eq6} in time, which justifies
testing it with the time-dependent function $v_h$. Integrating
by parts in time, we obtain
\begin{equation}
\label{eq36}
\begin{split}
&(u_t,v_h)\Big|_s^t
-\int_s^t(u_t,\partial_\tau v_h)\,d\tau
+\int_s^t(\nabla u,\nabla v_h)\,d\tau-\int_s^t\int_0^\tau
g(\tau-r)(\nabla u(r),\nabla v_h(\tau))\,dr\,d\tau\\
&\quad+\int_s^t\int_\Omega b(x,u_t)v_h\,dx\,d\tau
=\int_s^t\int_\Omega f(x,u)v_h\,dx\,d\tau.
\end{split}
\end{equation}

For a continuous function $z$ with values in $H$, defined on a neighborhood of $[s,t]$, a change of variables yields
\[
\begin{aligned}
&\int_s^t
\left(z(\tau),
\frac{z(\tau+h)-z(\tau-h)}{2h}\right)\,d\tau\\
&\quad=
\frac1{2h}\left[
\int_{t-h}^t(z(\tau),z(\tau+h))\,d\tau
-\int_{s-h}^s(z(\tau),z(\tau+h))\,d\tau
\right]\to
\frac12\|z\|_2^2\Big|_s^t \quad\text{as }h\to0.
\end{aligned}
\]
Applying the above identity to $z=u_t$ and $z=\nabla u$,
respectively, we obtain
\begin{equation}
\label{eq37}
(u_t,v_h)\Big|_s^t
-\int_s^t(u_t,\partial_\tau v_h)\,d\tau
\to\frac12\|u_t\|_2^2\Big|_s^t\quad\text{as }h\to0,
\end{equation}
\[
\int_s^t(\nabla u,\nabla v_h)\,d\tau
\to\frac12\|\nabla u\|_2^2\Big|_s^t\quad\text{as }h\to0.
\]

For the memory calculation, temporarily write
\[
z(\tau)=\nabla u(\tau),
\quad
k(\tau)=\int_0^\tau g(\tau-r)z(r)\,dr.
\]
Since $z$ is continuous and $g\in W^{1,1}(0,T)$, convolution
differentiation gives $k\in W^{1,1}(0,T;L^2(\Omega)^n)$, with
\begin{equation}
\label{eq38}
k'(\tau)
=g(0)z(\tau)+\int_0^\tau g'(\tau-r)z(r)\,dr
\quad\text{for a.e. }\tau.
\end{equation}
A change of variables in the difference quotient,
followed by passage to the limit, shows
\[
-\int_s^t
\left(k(\tau),
\frac{z(\tau+h)-z(\tau-h)}{2h}\right)\,d\tau
\to
-(k,z)\Big|_s^t+\int_s^t(k',z)\,d\tau\quad\text{as }h\to0.
\]
A direct computation together with \eqref{eq38} implies
\begin{equation}
\label{eq39}
\begin{split}
&\left[\frac12\|z\|_2^2-(k,z)\right]_s^t
+\int_s^t(k',z)\,d\tau\\
&=\frac12\left[
\left(1-\int_0^\tau g(r)\,dr\right)\|z(\tau)\|_2^2
+\int_0^\tau g(\tau-r)\|z(\tau)-z(r)\|_2^2\,dr
\right]_s^t\\
&\quad+\frac12\int_s^t g(\tau)\|z(\tau)\|_2^2\,d\tau
+\frac12\int_s^t\int_0^\tau[-g'(\tau-r)]
\|z(\tau)-z(r)\|_2^2\,dr\,d\tau.
\end{split}
\end{equation}

Since $v_h\to u_t$ in
$L^{m(x)}(\Omega\times(s,t))$, the variable-exponent
H\"{o}lder's inequality yields
\[
\begin{aligned}
&\left|
\int_s^t\int_\Omega b(x,u_t)v_h\,dx\,d\tau
-\int_s^t\int_\Omega |u_t|^{m(x)}\,dx\,d\tau
\right|\\
&\quad=
\left|\int_s^t\int_\Omega b(x,u_t)(v_h-u_t)\,dx\,d\tau\right|\\
&\quad\le
2\|b(x,u_t)\|_{L^{m'(x)}(\Omega\times(s,t))}
\|v_h-u_t\|_{L^{m(x)}(\Omega\times(s,t))}
\to0,
\end{aligned}
\]
where $b(x,u_t)u_t=|u_t|^{m(x)}$. Similarly,
\[
\begin{aligned}
&\left|
\int_s^t\int_\Omega f(x,u)v_h\,dx\,d\tau
-\int_s^t\int_\Omega f(x,u)u_t\,dx\,d\tau
\right|\\
&\quad\le
2\|f(x,u)\|_{L^{m'(x)}(\Omega\times(s,t))}
\|v_h-u_t\|_{L^{m(x)}(\Omega\times(s,t))}
\to0.
\end{aligned}
\]

For almost every $x\in\Omega$, the function $u(x,\cdot)$
has an absolutely continuous representative with
derivative $u_t(x,\cdot)$. Since $\partial_rF(x,r)=f(x,r)$
and $f(x,u)u_t\in L^1(Q_T)$, the scalar chain rule and
Fubini's theorem give
\begin{equation}
\label{eq40}
\int_\Omega F(x,u(t))\,dx-\int_\Omega F(x,u(s))\,dx
=\int_s^t\int_\Omega f(x,u)u_t\,dx\,d\tau
\end{equation}
for $s,t$ in a common set of full measure in $(0,T)$,
with $s\le t$. 
Moreover, with $p_++\sigma<2^*$, the source growth
estimate yields
\[
\int_\Omega|F(x,a)-F(x,b)|\,dx
\le C\|a-b\|_{p_++\sigma}
\left(1+\|a\|_{p_++\sigma}^{{p_++\sigma}-1}+\|b\|_{p_++\sigma}^{{p_++\sigma}-1}\right),
\quad a,b\in L^{p_++\sigma}(\Omega).
\]
Thus $u\in C([0,T];V)$ ensures continuity of the
integrated potential. Therefore, \eqref{eq40} holds for every $0\le s\le t\le T$.

Combining \eqref{eq36} with \eqref{eq37}, \eqref{eq39} and \eqref{eq40} proves \eqref{eq7}
for $0<s<t<T$. Continuity of the energy terms and
integrability of the dissipations extend the identity
to all $0\le s\le t\le T$. In particular, $E$ is locally absolutely continuous and,
for almost every $t$,
\begin{equation}\label{eq41}
\begin{aligned}
E'(t)
&=-\int_\Omega|u_t(x,t)|^{m(x)}\,dx
-\frac12g(t)\|\nabla u(t)\|_2^2\\
&\quad-\frac12\int_0^t[-g'(t-s)]
\|\nabla u(t)-\nabla u(s)\|_2^2\,ds
\le0.
\end{aligned}
\end{equation}

\subsection{Continuation}
Suppose that an energy solution is defined on
$[0,T)$, with $T<\infty$, and
\[
\sup_{0\le t<T}
\bigl(\|u(t)\|_V+\|u_t(t)\|_2\bigr)\le R_T.
\]
By \eqref{1403}, one has 
\begin{align*}
\|\eta^t\|_{\mathcal M}^2
=\int_0^t g(t-s)
\|\nabla u(t)-\nabla u(s)\|_2^2\,ds
+\int_t^\infty g(s)\|\nabla u(t)\|_2^2\,ds\le4(1-\ell)R_T^2.
\end{align*}
Hence $\sup_{t<T}\|U(t)\|_{\HH}<\infty$.
Moreover, the source growth estimate in
Lemma~\ref{lem:growth}, together with
$m'(x)\le m_-'$ and
$q_\sigma=(p_+-1+\sigma)m_-'<2^*$,
implies, by the Sobolev embedding
$V\hookrightarrow L^{q_\sigma}(\Omega)$, that
\[
\int_0^T\int_\Omega |f(x,u)|^{m'(x)}\,dx\,dt
\le C\int_0^T
\bigl(1+\|u(t)\|_{q_\sigma}^{q_\sigma}\bigr)\,dt
<\infty.
\]

We next establish a terminal value for the full phase.
Regard $f(x,u)$ as a prescribed forcing and choose
$h_j\in C_c^\infty(Q_T)$ such that
\[
h_j\to f(x,u)
\quad\text{strongly in }L^{m'(x)}(Q_T)\quad\text{as }j\to\infty.
\]
Let
$W_j=(\widetilde u_j,\widetilde v_j,\widetilde\eta_j)$
solve
\[
W_j'+\Op W_j=(0,h_j,0),
\quad W_j(0)=U_0.
\]
For each $j$, approximate $U_0$ by initial data in $D(\Op)$
and apply the energy estimate to the corresponding strong
solutions of the problem forced by $h_j$. Passing to the
limit in the initial data, and then applying Young's
inequality \eqref{eq4} as in \eqref{eq33}, gives
\[
\sup_j\left[
\sup_{0\le t\le T}\|W_j(t)\|_{\HH}^2
+\int_{Q_T}|\widetilde v_j|^{m(x)}\,dx\,dt
+\int_0^T\int_0^\infty[-g'(s)]
\|\nabla\widetilde\eta_j(t,s)\|_2^2\,ds\,dt
\right]<\infty.
\]
Repeating the derivation of the difference energy
identity \eqref{eq32} for the prescribed forcings
$h_j$ and $h_k$, and using the damping monotonicity
estimate \eqref{eq30} and the variable-exponent
H\"older's inequality \eqref{eq3}, we obtain
\begin{align*}
&\frac12\|W_j(t)-W_k(t)\|_{\HH}^2
+2^{2-m_+}\int_0^t\int_\Omega
|\widetilde v_j-\widetilde v_k|^{m(x)}
\,dx\,d\tau\\
&\quad+\frac12\int_0^t\int_0^\infty[-g'(s)]
\|\nabla(\widetilde\eta_j-\widetilde\eta_k)(\tau,s)\|_2^2
\,ds\,d\tau\\
&\le
2\|h_j-h_k\|_{L^{m'(x)}(Q_T)}
\|\widetilde v_j-\widetilde v_k\|_{L^{m(x)}(Q_T)}
\to0\quad\text{as }j,k\to\infty,
\end{align*}
uniformly for $t\in[0,T]$.
Thus, as $j\to\infty$, 
\[W_j\to W \quad\text{strongly in } C([0,T];\HH),\] 
with strong
convergence in the damping and history-dissipation
spaces. Therefore, $W$ solves the
prescribed-forcing problem with forcing $f(x,u)$.

The given solution $U$ and $W$ have the same initial
data and prescribed forcing. Repeating the
symmetric-time-difference argument in
\eqref{eq36}--\eqref{eq39} for their difference
on compact subintervals of $[0,T)$, and then
letting the lower endpoint tend to zero, gives
\begin{align*}
&\frac12\|U(t)-W(t)\|_{\HH}^2
+\int_0^t\int_\Omega
[b(x,u_t)-b(x,\widetilde v)]
(u_t-\widetilde v)\,dx\,d\tau\\
&\quad+\frac12\int_0^t\int_0^\infty[-g'(s)]
\|\nabla(\eta-\widetilde\eta)(\tau,s)\|_2^2
\,ds\,d\tau=0,
\end{align*}
where $W=(\widetilde u,\widetilde v,\widetilde\eta)$.
By \eqref{eq30} and $-g'\ge0$ from
Assumption~\ref{ass1}(2), all terms are nonnegative,
so $U=W$ on $[0,T)$.
Consequently,
\[
U(t)\to U_T:=W(T)\quad\text{in }\HH
\text{ as }t\to T.
\]

The local construction leading to \eqref{eq34}
and the subsequent passage to the original source
applies to the terminal phase $U_T\in\HH$, including
its history component $\eta^T\in\mathcal M$.
Writing $\widehat u(t)=u(T+t)$, the continued history is
\[
\widehat\eta^t(s)=
\begin{cases}
\widehat u(t)-\widehat u(t-s),&0<s\le t,\\
\eta^T(s-t)+\widehat u(t)-\widehat u(0),&s>t.
\end{cases}
\]
Thus the previous history is retained, and the new
solution joins continuously to the original phase
at $T$, extending the solution beyond $T$.

To obtain a maximal extension, order all extensions of the
given local solution by extension of both their time intervals
and solution values. Let $\mathcal C$ be a chain of such
extensions, and let $[0,T_*)$ be the union of their time
intervals. For every $S<T_*$, some member of $\mathcal C$ is
defined beyond $S$, and all members defined on $[0,S]$ agree
there. Their union is therefore an energy solution on each
compact subinterval of $[0,T_*)$ and provides an upper bound
for $\mathcal C$. Zorn's lemma then yields a maximal extension
on $[0,T_{\max})$.

If $T_{\max}<\infty$ and
\[
\limsup_{t\uparrow T_{\max}}
\bigl(\|u(t)\|_V+\|u_t(t)\|_2\bigr)<\infty,
\]
then continuity on compact time intervals implies
that these norms are bounded throughout
$[0,T_{\max})$. The preceding continuation argument
would extend the solution beyond $T_{\max}$,
contradicting maximality. Hence
\[
\limsup_{t\uparrow T_{\max}}
\bigl(\|u(t)\|_V+\|u_t(t)\|_2\bigr)=\infty
\quad\text{whenever }T_{\max}<\infty.
\]

\section{Finite-time blow-up}\label{sec:blowproof}

The following lemma was established in our previous work \cite{LZ}, to which we refer the reader for its proof.
\begin{lemma}\label{correctionlemma}
For every $q\ge p_->2$ and $r\in\mathbb R$,
\begin{equation}\label{eq42}
|r|^q\left(
\frac{q-p_-}{q}\log|r|+\frac{p_-}{q^2}
\right)\ge -c_{p_-}(q),
\end{equation}
where the expression at $r=0$ is defined by continuity.
For each fixed $q$, $c_{p_-}(q)$ is the smallest
nonnegative constant for which
\eqref{eq42} holds for all $r\in\mathbb R$.
Consequently, every energy solution satisfies
\begin{equation}\label{eq43}
\begin{aligned}
\int_\Omega |u|^{p(x)}\log|u|\,dx
\ge{}&
\frac{p_-}{2}\norm{u_t}_2^2
+\frac{p_-}{2}
\left(1-\int_0^t g(s)\,ds\right)
\norm{\nabla u}_2^2\\
&+\frac{p_-}{2}\int_0^t g(t-s)\|\nabla u(t)-\nabla u(s)\|_2^2\,ds
-p_-\E(t).
\end{aligned}
\end{equation}
\end{lemma}

We next establish some preliminary inequalities.

\begin{lemma}\label{lem:pairing-estimate}
Assume that Assumptions~\ref{ass1} and~\ref{ass2} hold. For almost every
$t$,
\begin{equation}\label{eq47}
\begin{split}
\frac d{dt}(u,u_t)
\ge{}&
\left(1+\frac {\hat{q}}{2}\right)\|u_t\|_2^2
+\left(\frac{\hat{q}-2}{2}\ell-\frac{1-\ell}{4\mu}\right)\|\nabla u\|_2^2\\
&+\left(1-\frac {\hat{q}}{p_-}\right)
\int_\Omega|u|^{p(x)}\log|u|dx
-\hat{q}\E(t)-\left|\int_\Omega b(x,u_t)udx\right|,
\end{split}
\end{equation}
where $\hat{q},\mu$ is shown as \eqref{eq46}.
\end{lemma}
\begin{proof}
Since $u\in C([0,T];V)$ and $m_+<2^*$, the Sobolev embedding
$V\hookrightarrow L^{2^*}(\Omega)$ yields
\[
\int_{Q_T}|u(x,t)|^{m(x)}\,dx\,dt
\le T|\Omega|+\int_0^T\|u(t)\|_{2^*}^{2^*}\,dt
\le T|\Omega|+TC^{2^*}
\sup_{0\le t\le T}\|\nabla u(t)\|_2^{2^*}<\infty.
\]
Thus
$u\in L^{m(x)}(Q_T)$. 
We may therefore test the weak equation with $u$ after
time regularization. Applying the product rule to $(u,u_t)$ leads  to 
\begin{equation}
\label{eq44}
\begin{split}
\frac d{dt}(u,u_t)
&=\|u_t\|_2^2
-\left(1-\int_0^tg(s)ds\right)\|\nabla u\|_2^2+\int_0^tg(t-s)(\nabla u(t),\nabla u(s)-\nabla u(t))ds\\
&\quad-\int_\Omega b(x,u_t)udx
+\int_\Omega f(x,u)u\,dx.
\end{split}
\end{equation}
For any $\mu>0$, Cauchy--Schwarz and $ab\le a^2/(4\mu)+\mu b^2$ results in 
\begin{equation}
\label{eq45}
\begin{split}
&\left|\int_0^tg(t-s)(\nabla u(t),\nabla u(s)-\nabla u(t))ds\right|\\
&\quad\le\left(\int_0^tg(s)ds\right)^{1/2}\|\nabla u(t)\|_2
\left(\int_0^tg(t-s)\|\nabla u(t)-\nabla u(s)\|_2^2ds\right)^{1/2}\\
&\quad\le\frac{1-\ell}{4\mu}\|\nabla u(t)\|_2^2
+\mu\int_0^tg(t-s)\|\nabla u(t)-\nabla u(s)\|_2^2ds.
\end{split}
\end{equation}
We may choose
$\hat q\in(2,p_-)$ sufficiently close to $p_-$ such that
$\ell>1/(\hat q-1)^2$, which implies $\frac{\hat{q}-2}{2}\ell-\frac{1-\ell}{2\hat{q}}>0.
$
Choose $\mu<\hat{q}/2$ sufficiently close to $\hat{q}/2$ so that
\begin{equation}\label{eq46}
\frac{\hat{q}-2}{2}\ell-\frac{1-\ell}{4\mu}>0,
\quad \frac {\hat{q}}{2}-\mu>0.
\end{equation}
Split the source integral into the fractions $\hat q/p_-$ and
$1-\hat q/p_-$. Multiplying \eqref{eq43} by $\hat q/p_-$ establishes
\[
\begin{aligned}
\int_\Omega f(x,u)u\,dx
\ge{}&\frac{\hat q}{2}\|u_t\|_2^2
+\frac{\hat q}{2}\left(1-\int_0^t g(s)\,ds\right)\|\nabla u\|_2^2+\frac{\hat q}{2}\int_0^t g(t-s)\|\nabla u(t)-\nabla u(s)\|_2^2\,ds
\\
&-\hat q\E(t)+\left(1-\frac{\hat q}{p_-}\right)
\int_\Omega|u|^{p(x)}\log|u|\,dx.
\end{aligned}
\]
Combining this with \eqref{eq44}--\eqref{eq45} and
$1-\int_0^t g\,ds\ge\ell$ yields \eqref{eq47}.
\end{proof}

\begin{lemma}\label{lem:Y-estimate}
Under \eqref{eq9}, define
\[
Y(t)=H^{1-\alpha}(t)+\varepsilon(u(t),u_t(t)),
\]
where $H(t)\ge h_*>0$ is an auxiliary energy function and
$\varepsilon>0$ is a constant. Their precise definitions will be
given in the proofs of Theorems~\ref{thm:negative} and~\ref{thm:high}.
Then
\begin{equation}\label{eq53}
Y^r(t)\le A\left[
H(t)+\|u_t(t)\|_2^2+\|\nabla u(t)\|_2^2
+\frac{|\Omega|}{ep_-}
+\int_\Omega |u|^{p(x)}\log|u|\,dx
\right],
\end{equation}

If, in addition,
\[
\int_\Omega |u|^{p(x)}\log|u|\,dx\ge0,
\]
then
\begin{equation}\label{eq54}
Y^r(t)\le A\left[
H(t)+\|u_t(t)\|_2^2+\|\nabla u(t)\|_2^2
+\int_\Omega |u|^{p(x)}\log|u|\,dx
\right].
\end{equation}
Here $A$ is given in \eqref{1517}.
\end{lemma}

\begin{proof}
For any $x,y\ge0$, $k\ge2$ and $a>0$, Young's inequality shows
\[
xy^{k-1}\le ay^k+
\frac{(k-1)^{k-1}}{k^k}a^{-(k-1)}x^k
\le ay^k+a^{-(k-1)}x^k.
\]
Take $k=m(x)$, $a=\lambda H^{-\alpha}$, and $\lambda\ge1$. Since
$\lambda^{-(m(x)-1)}\le\lambda^{-1}$, we have 
\begin{equation}\label{eq48}
\begin{aligned}
\left|\int_\Omega b(x,u_t)udx\right|
\le\lambda H^{-\alpha}\int_\Omega|u_t|^{m(x)}dx+\lambda^{-1}\int_\Omega H^{\alpha(m(x)-1)}|u|^{m(x)}dx.
\end{aligned}
\end{equation}
To estimate the second term in the right-hand side of \eqref{eq48}, split $\Omega$ into
$\{|u|\le e\}$ and $\{|u|>e\}$.
For  $\{|u|>e\}$, we have $\log|u|>1$ and
$|u|^{p_-}\le |u|^{p(x)}$. Therefore, 
\begin{equation}
\label{eq49}
\begin{split}
\int_{\{|u|>e\}}|u|^{p_-}\,dx
&\le \int_{\{|u|>e\}}|u|^{p(x)}\log|u|\,dx\\
&=\int_\Omega |u|^{p(x)}\log|u|\,dx
-\int_{\{1\le |u|\le e\}}|u|^{p(x)}\log|u|\,dx\\
&\quad+\int_{\{|u|<1\}}|u|^{p(x)}|\log|u||\,dx\\
&\le \frac{|\Omega|}{ep_-}+\int_\Omega |u|^{p(x)}\log|u|\,dx.
\end{split}
\end{equation}
Here  we used 
\[
\int_{\{|u|<1\}}|u|^{p(x)}|\log|u||\,dx
\le\int_{\{|u|<1\}}\frac1{ep(x)}\,dx
\le\frac{|\Omega|}{ep_-}\]
by \eqref{eq22}.

Since $|u|^{m(x)}\le e^{m_+}$ on $\{|u|\le e\}$,
H\"{o}lder's inequality and \eqref{eq49} lead to 
\begin{equation}
\label{eq50}
\begin{split}
\int_\Omega |u|^{m(x)}\,dx
&\le e^{m_+}|\Omega|
+\int_{\{|u|>e\}}|u|^{m_+}\,dx\\
&\le e^{m_+}|\Omega|
+|\Omega|^{1-\beta}
\left(\int_{\{|u|>e\}}|u|^{p_-}\,dx\right)^\beta\\
&\le e^{m_+}|\Omega|
+|\Omega|^{1-\beta}
\left(\frac{|\Omega|}{ep_-}+\int_\Omega |u|^{p(x)}\log|u|\,dx\right)^\beta\\
&\le C_1
\left(1+\frac{|\Omega|}{ep_-}+\int_\Omega |u|^{p(x)}\log|u|\,dx\right)^\beta,
\end{split}
\end{equation}
where $C_1=e^{m_+}|\Omega|+|\Omega|^{1-\beta}$, $\beta=m_+/p_-<1$.
Similarly,
\begin{equation}\label{eq51}
\|u\|_2\le C_2
\left(1+\frac{|\Omega|}{ep_-}+\int_\Omega|u|^{p(x)}\log|u|\,dx\right)^{1/p_-},
\quad C_2=\left(e^2|\Omega|+|\Omega|^{1-2/p_-}\right)^{1/2}.
\end{equation}
Since $r/p_-+r/2<1$ by recalling \eqref{eq9},
Cauchy--Schwarz, \eqref{eq51}, and generalized Young's inequality with weights $r/p_-$, $r/2$, and $1-r/p_--r/2$ illustrate 
\begin{align*}
|(u,u_t)|^r
&\le\|u\|_2^r\|u_t\|_2^r\\
&\le C_2^r\left(1+\frac{|\Omega|}{ep_-}+\int_\Omega|u|^{p(x)}\log|u|\,dx\right)^{r/p_-}
\|u_t\|_2^r\\
&\le C_2^r\left(1+\frac{|\Omega|}{ep_-}+\int_\Omega|u|^{p(x)}\log|u|\,dx
+\|u_t\|_2^2\right).
\end{align*}
Here  we used 
\begin{equation}
\label{eq52}
\int_\Omega|u|^{p(x)}\log|u|\,dx
=\int_{\{|u|\ge1\}}|u|^{p(x)}\log|u|\,dx-\int_{\{|u|<1\}}|u|^{p(x)}|\log|u||\,dx
\ge-\frac{|\Omega|}{ep_-}.
\end{equation}

Since
$r(1-\alpha)=1$ by recalling \eqref{eq9}, the preceding estimate entails
\[
\begin{aligned}
Y^r(t)
&\le2^{r-1}\left(H(t)+\varepsilon^r|(u,u_t)|^r\right)\\
&\le2^{r-1}\left[
H(t)+\varepsilon^rC_2^r
\left(1+\frac{|\Omega|}{ep_-}+\int_\Omega|u|^{p(x)}\log|u|\,dx
+\|u_t\|_2^2\right)\right].
\end{aligned}
\]
Using $1\le H(t)/h_*$ and \eqref{eq52}, we obtain \eqref{eq53}
\begin{equation}\label{1517}
A=2^{r-1}\left(
1+\varepsilon^rC_2^r
\max\left\{1,\frac{1+\frac{|\Omega|}{ep_-}}{h_*}\right\}
\right).
\end{equation}
If the source integral is nonnegative, absorbing $1+\frac{|\Omega|}{ep_-}$
into $H(t)$ due to $1+\frac{|\Omega|}{ep_-}\le(1+\frac{|\Omega|}{ep_-})H(t)/h_*$ instead gives \eqref{eq54}, with the same constant $A$. 

In fact, \eqref{eq54} applies when $\E(t)<0$, otherwise,
we retain form \eqref{eq53}.
\end{proof}

\subsection{Proof of Theorem~\ref{thm:negative}}\label{sec:negative}
We begin with strictly negative shifted energy and then treat the critical
case by starting the differential inequality at a positive time.
First assume $\E(0)<0$ and define 
\[H(t)=-\E(t),\quad h_0=H(0)>0.\]
It follows from \eqref{eq41} and \eqref{eq43} that 
\begin{equation}\label{eq55}
H(t)\ge h_0,\quad
H'(t)\ge\int_\Omega|u_t|^{m(x)}dx,\quad
H(t)\le\frac1{p_-}\int_\Omega|u|^{p(x)}\log|u|dx.
\end{equation}
 Choose $\mu$ by
\eqref{eq46} and put
\[
c_0=\min\left\{1+\frac {\hat{q}}{2},
\frac{\hat{q}-2}{2}\ell-\frac{1-\ell}{4\mu},
1-\frac {\hat{q}}{p_-},\hat{q}\right\}>0.
\]
By \eqref{eq47}, one obtains 
\begin{equation}
\label{eq56}
\frac d{dt}(u,u_t)\ge c_0\left[H(t)+\|u_t\|_2^2+\|\nabla u\|_2^2
+\int_\Omega|u|^{p(x)}\log|u|dx\right]-\left|\int_\Omega b(x,u_t)udx\right|.
\end{equation}

To estimate the second term on the right-hand side of \eqref{eq48},
set $d=\alpha(m_+-1)$ and $\beta=m_+/p_-$. By the choice of
$\alpha$, $d+\beta<1$. Since $H(t)\ge h_0$,
\[
H^{\alpha(m(x)-1)}(t)\le L_0H^d(t),\]
with $L_0=\max\left\{1,h_0^{-\alpha(m_+-m_-)}\right\}.$
Indeed, for $H(t)\ge1$ this follows from $\alpha(m(x)-1)\le d$,
while for $h_0\le H(t)<1$ it follows by dividing by $H^d(t)$.
Using \eqref{eq50}, \eqref{eq55}, and
$(a+b)^\beta\le a^\beta+b^\beta$, we obtain
\begin{equation}
\label{eq58}
\begin{split}
\int_\Omega H(t)^{\alpha(m(x)-1)}|u|^{m(x)}\,dx
&\le L_0C_1p_-^{-d}
\bigg[
(1+\frac{|\Omega|}{ep_-})^\beta
\left(\int_\Omega|u|^{p(x)}\log|u|\,dx\right)^d\\
&\qquad+
\left(\int_\Omega|u|^{p(x)}\log|u|\,dx\right)^{d+\beta}
\bigg]\\
&\le C_3\int_\Omega|u|^{p(x)}\log|u|\,dx,
\end{split}
\end{equation}
where
\[
C_3=L_0C_1p_-^{-d}
\left[
(1+\frac{|\Omega|}{ep_-})^\beta(p_-h_0)^{d-1}
+(p_-h_0)^{d+\beta-1}
\right].
\]
The last inequality follows from
$\int_\Omega|u|^{p(x)}\log|u|\,dx\ge p_-h_0>0$ by \eqref{eq55}
and $d+\beta<1$, since both powers $d-1$ and
$d+\beta-1$ are negative.

Take $\lambda=\max\{1,2C_3/c_0\}$ and $c=c_0/2$. Combining \eqref{eq56} with \eqref{eq48} and \eqref{eq58}, we obtain 
\begin{equation}
\label{eq59}
\frac d{dt}(u,u_t)\ge c\left[H(t)+\|u_t\|_2^2+\|\nabla u\|_2^2
+\int_\Omega|u|^{p(x)}\log|u|dx\right]-\lambda H^{-\alpha}(t)\int_\Omega|u_t|^{m(x)}dx.
\end{equation}
Choose 
\[
0<\varepsilon\le\min\left\{\frac{1-\alpha}{2\lambda},
\frac{h_0^{1-\alpha}}{2(1+|(u_0,u_1)|)}\right\}.
\]
For $Y(t)=H^{1-\alpha}(t)+\varepsilon(u,u_t)$ this gives
$Y(0)\ge h_0^{1-\alpha}/2>0$ and by \eqref{eq54} with $h_*=h_0$ and \eqref{eq59}, we obtain 
\begin{align*}
Y'(t)&=(1-\alpha)H^{-\alpha}H'+\varepsilon\frac d{dt}(u,u_t)\\
&\ge(1-\alpha-\varepsilon\lambda)H^{-\alpha}
\int_\Omega|u_t|^{m(x)}dx\\
&\quad+\varepsilon c\left[H+\|u_t\|_2^2+\|\nabla u\|_2^2
+\int_\Omega|u|^{p(x)}\log|u|dx\right]\\
&\ge\frac{\varepsilon c}{A}Y^r(t).
\end{align*}
On any interval where $Y(t)>0$, the above inequality gives
$Y'(t)\ge(\varepsilon c/A)Y^r(t)>0$ almost everywhere.
If $Y(t)$ had a first zero $t_*<T_{\max}$, integration on
$[0,t]$, followed by $t\to t_*$, would yield
$Y(t_*)\ge Y(0)>0$, a contradiction.
Thus $Y(t)>0$ throughout $[0,T_{\max})$. 
Multiplication by $(1-r)Y^{-r}<0$ and integration yield
\[
\frac d{dt}Y^{1-r}\le-(r-1)\frac{\varepsilon c}{A},\]
further, 
\[
0<Y(t)^{1-r}\le Y(0)^{1-r}-(r-1)\frac{\varepsilon c}{A}t.\]
Set
\[
T_*=\frac{A}{\varepsilon c(r-1)}Y(0)^{1-r}.
\]
If $T_{\max}>T_*$, then $Y(T_*)$ is finite and positive,
whereas the preceding estimate gives
\[
0<Y(T_*)^{1-r}
\le Y(0)^{1-r}
-\frac{\varepsilon c(r-1)}{A}T_*=0,
\]
a contradiction. Hence $T_{\max}\le T_*<\infty$.
The continuation criterion in Theorem~\ref{thm:local}
then implies
\[
\limsup_{t\uparrow T_{\max}}
\bigl(\|u(t)\|_V+\|u_t(t)\|_2\bigr)=\infty.
\]

Now assume $\E(0)=0$ and $(u_0,u_1)\ne(0,0)$. It follows from 
$\E(t)=E(t)+\frac{C_*}{p_-}$ and \eqref{eq7} that 
\[
\begin{aligned}
\E(t)
={}&-\int_0^t\int_\Omega|u_\tau|^{m(x)}\,dx\,d\tau
-\frac12\int_0^t g(\tau)\|\nabla u(\tau)\|_2^2\,d\tau\\
&-\frac12\int_0^t\int_0^\tau[-g'(\tau-s)]
\|\nabla u(\tau)-\nabla u(s)\|_2^2\,ds\,d\tau
\le0.
\end{aligned}
\]
If $\mathcal E(t)=0$ for some $t>0$, then the energy identity
implies that all three dissipation integrals vanish. In particular,
\[
\int_0^t g(\tau)\|\nabla u(\tau)\|_2^2\,d\tau=0.
\]
Since $g(\tau)>0$, it follows that $\|\nabla u(\tau)\|_2=0$ for
almost every $\tau\in(0,t)$. The continuity
$u\in C([0,t];V)$ then gives $\|\nabla u(\tau)\|_2=0$ for every
$\tau\in[0,t]$. Hence $u\equiv0$ on $[0,t]$, and therefore
$u_t\equiv0$, contradicting $(u_0,u_1)\ne(0,0)$. Thus
$\mathcal E(\tau)<0$ for every $0<\tau<T_{\max}$.

Fix such a $\tau$ and set $h_\tau=-\E(\tau)>0$.
The preceding argument applies on $[\tau,T_{\max})$,
with $h_0$ replaced by $h_\tau$ and the initial pairing
replaced by $(u(\tau),u_t(\tau))$. Choosing the corresponding
constants $\varepsilon_\tau,c_\tau,A_\tau$ and setting
\[
Y_\tau(t)=(-\E(t))^{1-\alpha}
+\varepsilon_\tau(u(t),u_t(t)),
\]
we obtain $Y_\tau(\tau)>0$ and
\[
0<Y_\tau(t)^{1-r}
\le Y_\tau(\tau)^{1-r}
-\frac{\varepsilon_\tau c_\tau(r-1)}{A_\tau}(t-\tau).
\]
Therefore,
\[
T_{\max}\le
\tau+\frac{A_\tau}{\varepsilon_\tau c_\tau(r-1)}
Y_\tau(\tau)^{1-r}<\infty.
\]
All memory integrals retain their original lower limit
zero, and only the integration of the differential inequality
starts at $\tau$. The norm blow-up follows from
Theorem~\ref{thm:local}.

\subsection{Proof of Theorem~\ref{thm:high}}\label{sec:high}
In this section,  the signed source integral need not be positive. 
Let 
\begin{equation}
\label{eq60}
H(t)=1+\E(0)-\E(t).
\end{equation}
Then by \eqref{eq41}
\begin{equation}
\label{eq61}
H(t)\ge1,\quad H'(t)\ge\int_\Omega|u_t|^{m(x)}\,dx.
\end{equation} 
Using \eqref{eq43} and \eqref{eq52}, we obtain 
\begin{equation}
\label{eq62}
\begin{split}
H(t)&\le1+[\E(0)]_++\frac1{p_-}
\int_\Omega|u|^{p(x)}\log|u|\,dx\\
&=1+[\E(0)]_++\frac1{p_-}\left(
\frac{|\Omega|}{ep_-}+\int_\Omega|u|^{p(x)}\log|u|\,dx\right)
-\frac{|\Omega|}{ep_-^2}\\
&\le1+[\E(0)]_++\frac{|\Omega|}{ep_-}+
\int_\Omega|u|^{p(x)}\log|u|\,dx.
\end{split}
\end{equation}
Combining \eqref{eq47} with \eqref{eq60},  \eqref{eq62}  and the choice of $c_0$ implies
\begin{equation}
\label{eq63}
\begin{split}
\frac d{dt}(u,u_t)\ge{}&c_0\left[H(t)+\|u_t\|_2^2+\|\nabla u\|_2^2
+\frac{|\Omega|}{ep_-}+\int_\Omega|u|^{p(x)}\log|u|dx\right]\\
&-[\hat{q}+(1-\hat{q}/p_-)\frac{|\Omega|}{ep_-}](1+[\E(0)]_+)
-\left|\int_\Omega b(x,u_t)udx\right|.
\end{split}
\end{equation}
Apply \eqref{eq48} with $\lambda=1$. It follows from  $H(t)\ge1$, \eqref{eq50}
 and \eqref{eq62} that 
\begin{equation}
\label{eq64}
\begin{split}
\int_\Omega H^{\alpha(m(x)-1)}(t)|u|^{m(x)}\,dx
&\le C_1H^{\alpha(m_+-1)}(t)\Big(1+\frac{|\Omega|}{ep_-}+
\int_\Omega|u|^{p(x)}\log|u|\,dx\Big)^{m_+/p_-}\\
&\le C_1\Big(1+[\E(0)]_++\frac{|\Omega|}{ep_-}+
\int_\Omega|u|^{p(x)}\log|u|\,dx\Big)^\theta,
\end{split}
\end{equation}
where 
$\theta=\alpha(m_+-1)+m_+/p_-<1$. 
For $x\ge0$, we recall 
$
x^\theta\le\delta x+(1-\theta)(\theta/\delta)^{\theta/(1-\theta)}.
$ 
Take $\delta=c_0/(2C_1)$ and denote the last constant by $C_\theta$.
Note that $1+[\E(0)]_+\ge1$, then \eqref{eq64} shows 
\begin{equation}
\label{eq65}
\int_\Omega H^{\alpha(m(x)-1)}(t)|u|^{m(x)}\,dx\le\frac{c_0}{2}\left(\frac{|\Omega|}{ep_-}+\int_\Omega|u|^{p(x)}\log|u|dx\right)
+\left(\frac{c_0}{2}+C_1C_\theta\right)(1+[\E(0)]_+).
\end{equation}

Set 
\[
B=\hat{q}+(1-\hat{q}/p_-)\frac{|\Omega|}{ep_-}+\frac{c_0}{2}+C_1C_\theta,\quad
\varepsilon=\frac{1-\alpha}{2}.
\]
Using \eqref{eq61}, \eqref{eq63}, \eqref{eq65} and \eqref{eq48} with $\lambda=1$, we obtain
\begin{equation*}
\begin{split}
Y'(t)&=(1-\alpha)H^{-\alpha}(t)H'(t)+\varepsilon\frac d{dt}(u,u_t)\\
&\ge \varepsilon c\left[H(t)+\|u_t\|_2^2+\|\nabla u\|_2^2
+\frac{|\Omega|}{ep_-}+\int_\Omega|u|^{p(x)}\log|u|dx\right]
-\varepsilon B(1+[\E(0)]_+).
\end{split}
\end{equation*}
Combining the above inequality with \eqref{eq53} equipped with  $h_*=1$, we obtain, as long as $Y(t)>0$,
\[
Y'(t)\ge\frac{\varepsilon c}{A}Y^r(t)-\varepsilon B(1+[\E(0)]_+).
\]

By \eqref{eq10},
\[
Y(0)>
\left[\frac{2AB}{c}\bigl(1+[\E(0)]_+\bigr)\right]^{1/r}.
\]
Whenever $Y(t)$ exceeds this threshold,
\[
Y'(t)\ge\frac{\varepsilon c}{2A}Y^r(t)>0.
\]
A first-crossing argument therefore shows that $Y$
remains above the threshold throughout $[0,T_{\max})$.
Since $r>1$, integration gives
\[
0<Y(t)^{1-r}
\le Y(0)^{1-r}-\frac{\varepsilon c(r-1)}{2A}t.
\]
The right-hand side must remain positive for
$t<T_{\max}$, thus 
\[
T_{\max}\le
\frac{2A}{\varepsilon c(r-1)}Y(0)^{1-r}<\infty.
\]

\subsection{Proof of Theorem~\ref{thm:every}}
Choose $\varphi\in C_c^\infty(\Omega)$ such that
$0\le\varphi\le1$ and $\varphi=1$ on a ball
$B\Subset\Omega$ of positive measure. For each fixed $x$,
$F(x,\cdot)$ is even, decreases on $(0,1)$, and increases on
$(1,\infty)$. Hence its minimum is
$F(x,1)=-1/p(x)^2\ge-1/p_-^2$. 

For sufficiently large $s$ and every $x\in B$,
\[
F(x,s)=\frac{s^{p(x)}}{p(x)}
\left(\log s-\frac1{p(x)}\right)
\ge\frac{s^{p_-}}{p_+}
\left(\log s-\frac1{p_-}\right).
\]
Using the global lower bound for $F$ on $\Omega\setminus B$, we obtain
\[
\int_\Omega F(x,s\varphi)\,dx
\ge
\frac{|B|s^{p_-}}{p_+}
\left(\log s-\frac1{p_-}\right)
-\frac{|\Omega|}{p_-^2}.
\]
Since $p_->2$, it follows that
\[
\frac{1}{s^2}\int_\Omega F(x,s\varphi)\,dx\to+\infty
\quad\text{as }s\to\infty.
\]

Fix $R_0\in\mathbb R$. For sufficiently large $s$, define
\[
u_0=s\varphi,\quad u_1=a_s\varphi,
\]
where
\[
a_s=
\left[
\frac{2}{\|\varphi\|_2^2}
\left(
R_0+\int_\Omega F(x,s\varphi)\,dx
-\frac{s^2}{2}\|\nabla\varphi\|_2^2
\right)
\right]^{1/2}.
\]
The quantity under the square root is positive for all sufficiently
large $s$. Since the memory term vanishes at $t=0$, direct substitution
into the energy gives
\[
E(0)=\frac12a_s^2\|\varphi\|_2^2
+\frac{s^2}{2}\|\nabla\varphi\|_2^2
-\int_\Omega F(x,s\varphi)\,dx
=R_0.
\]
Moreover,
\[
\frac{a_s^2}{s^2}
=\frac{2}{\|\varphi\|_2^2}
\left[
\frac{R_0}{s^2}
+\frac{1}{s^2}\int_\Omega F(x,s\varphi)\,dx
-\frac12\|\nabla\varphi\|_2^2
\right]\to+\infty\quad\text{as }s\to\infty.
\]
Thus $a_s/s\to\infty$, and consequently
\[
(u_0,u_1)=s a_s\|\varphi\|_2^2\to+\infty\quad\text{as }s\to\infty.
\]
The shifted initial energy $\mathcal E(0)=R_0+C_*/p_-$ is fixed.
Therefore, for sufficiently large $s$, the initial-pairing
condition \eqref{eq10} holds. Theorem~\ref{thm:high} then proves
Theorem~\ref{thm:every}.

\section{Global existence in the stable region}\label{sec:globalproof}
In this section, we prove Theorem \ref{thm:global} by showing that the initial-data condition
\eqref{eq13} confines the solution to an invariant stable region,
where the energy controls the velocity, displacement gradient, and
memory term.

It follows from \eqref{eq22} that 
\begin{equation}\label{eq66}
\begin{aligned}
\int_\Omega f(x,u)u\,dx
&\le\int_{|u|>1}|u|^{p(x)}\log|u|dx \le\frac1{e\sigma}\int_{|u|>1}|u|^{p(x)+\sigma}dx\\
&\le\frac1{e\sigma}\|u\|_{p_++\sigma}^{p_++\sigma}
\le\frac{S_p^{p_++\sigma}}{e\sigma}\|\nabla u\|_2^{p_++\sigma}.
\end{aligned}
\end{equation}
If $\|\nabla u\|_2\le R$, the definition of $R$ in \eqref{eq12} and the sign of $F$
on $|u|\le1$ lead to 
\begin{align*}
\int_\Omega F(x,u)\,dx
&\le\frac1{p_-}\int_{|u|>1}|u|^{p(x)}\log|u|\,dx\\
&\le \frac1{p_-}\frac{S_p^{p_++\sigma}}{e\sigma}\|\nabla u\|_2^{p_++\sigma}\\
&\le\frac1{p_-}\frac{S_p^{p_++\sigma}R^{{p_++\sigma}-2}}{e\sigma}\|\nabla u\|_2^2
=\frac\ell{2p_-}\|\nabla u\|_2^2.
\end{align*}
Using $1-\int_0^tg\,ds\ge\ell$ in the energy definition \eqref{eq5}, we obtain
\begin{equation}\label{eq67}
\begin{aligned}
E(t)\ge\frac12\|u_t(t)\|_2^2
+\frac{\ell(p_--1)}{2p_-}\|\nabla u(t)\|_2^2+\frac12\int_0^tg(t-s)\|\nabla u(t)-\nabla u(s)\|_2^2ds\ge c_E \widehat{E}(t).
\end{aligned}
\end{equation}
with \[c_E=\min\left\{\frac12,\frac{\ell(p_--1)}{2p_-}\right\}.\]
as long as $\|\nabla u(t)\|_2\le R$.

Next, we show that
\begin{equation}
\label{eq68}
\|\nabla u(t)\|_2<R
\quad\text{for every }t\in[0,T_{\max}).
\end{equation}
Suppose otherwise. Since $\|\nabla u_0\|_2<R$ and
$u\in C([0,T_{\max});V)$, there exists a first time
$t_*\in(0,T_{\max})$ such that
\[
\|\nabla u(t)\|_2<R\quad\text{for }0\le t<t_*,
\quad
\|\nabla u(t_*)\|_2=R.
\]
Thus \eqref{eq67} holds throughout $[0,t_*]$, which together with \eqref{eq41} and \eqref{eq12} derive 
\[
d_0
=\frac{\ell(p_--1)}{2p_-}R^2
=\frac{\ell(p_--1)}{2p_-}\|\nabla u(t_*)\|_2^2
\le E(t_*)\le E(0)<d_0,
\]
a contradiction. Consequently, the solution remains
in the stable region throughout its maximal interval
of existence.

In particular, \eqref{eq67} and $E(t)\le E(0)$ show
\[
\|u_t(t)\|_2^2\le 2E(0),
\quad
\|\nabla u(t)\|_2^2
\le\frac{2p_-}{\ell(p_--1)}E(0),
\quad 0\le t<T_{\max}.
\]
Hence $\|u(t)\|_V+\|u_t(t)\|_2$ is uniformly bounded
on $[0,T_{\max})$. If $T_{\max}<\infty$, this would
contradict the continuation criterion
\eqref{eq8}. Therefore $T_{\max}=\infty$,
and the solution is global.

Using again \eqref{eq22}, for  $0<|r|\le1$, then
\begin{align*}
\frac{[-F(x,r)]_+}{|r|^2}
&=\frac{|r|^{p(x)-2}|\log|r||}{p(x)}
 +\frac{|r|^{p(x)-2}}{p(x)^2}\\
&\le\frac{|r|^{p_--2}|\log|r||}{p_-}
 +\frac{|r|^{p_--2}}{p_-^2}
\le\frac1{ep_-(p_--2)}+\frac1{p_-^2}.
\end{align*}
For $|r|>1$, the negative part of $F$ vanishes unless
$|r|<e^{1/p(x)}\le e^{1/p_-}$. In this interval,
\[
\frac{[-F(x,r)]_+}{|r|^2}
\le\frac{|r|^{p(x)-2}}{p(x)^2}
\le\frac{e^{(p_+-2)/p_-}}{p_-^2}.
\]
Since $F(x,0)=0$, we conclude that
\[
-F(x,r)\le[-F(x,r)]_+\le C_4|r|^2,\]
with
\[
C_4=\max\left\{
\frac1{ep_-(p_--2)}+\frac1{p_-^2},
\frac{e^{(p_+-2)/p_-}}{p_-^2}
\right\}.
\]
Thus, Poincar\'e's inequality implies 
\[
-\int_\Omega F(x,u)\,dx
\le C_4\|u\|_2^2
\le C_4C_p^2\|\nabla u\|_2^2.
\]
Using $1-\int_0^t g(s)\,ds\le1$ in the energy definition \eqref{eq5},
we obtain
\begin{align*}
E(t)\le
\frac12\|u_t(t)\|_2^2
+\left(\frac12+C_4C_p^2\right)\|\nabla u(t)\|_2^2+\frac12\int_0^t g(t-s)
\|\nabla u(t)-\nabla u(s)\|_2^2\,ds\le C_E \widehat{E}(t),
\end{align*}
with
\[
C_E=\frac12+C_4C_p^2,
\]
which together with \eqref{eq67} completes the proof of \eqref{eq14}.

\section{Decay in the Sobolev damping range}\label{sec:decayproof}
\begin{lemma}\label{lem:firstmult}
Under the hypotheses of Theorem \ref{thm:global} and $m_+\le2^*$,
for almost every $t>0$,
\begin{equation}\label{eq69}
\frac d{dt}(u,u_t)
\le\|u_t\|_2^2-\frac\ell4\|\nabla u\|_2^2
+C\|\nabla w\|_2^2
+C\int_\Omega|u_t|^{m(x)}\,dx,
\end{equation}
where 
\[
w(t):=\int_0^t g(t-s)[u(t)-u(s)]\,ds.
\]
\end{lemma}

\begin{proof}
By Cauchy--Schwarz, \eqref{eq67}, and \eqref{eq41}, one establishes 
\begin{equation}\label{eq70}
\begin{aligned}
\|\nabla w(t)\|_2^2
&\le\left(\int_0^t g(s)\,ds\right)
\int_0^t g(t-s)
\|\nabla u(t)-\nabla u(s)\|_2^2\,ds\\
&\le(1-\ell)\int_0^t g(t-s)
\|\nabla u(t)-\nabla u(s)\|_2^2\,ds\\
&\le\frac{1-\ell}{c_E}E(t)
\le\frac{1-\ell}{c_E}E(0).
\end{aligned}
\end{equation}

It follows from \eqref{eq68} and \eqref{eq70}
that, for $z=u(t)$ or $z=w(t)$,
\begin{align*}
\int_\Omega|z|^{m(x)}\,dx
\le\|z\|_{m_-}^{m_-}+\|z\|_{m_+}^{m_+}
\le C^{m_-}\|\nabla z\|_2^{m_-}
+C^{m_+}\|\nabla z\|_2^{m_+}
\le C\|\nabla z\|_2^2,
\end{align*}
by splitting into $|z|\le1$ and $|z|>1$.
By Young's inequality and the above inequality, we get
\begin{equation}\label{eq71}
\begin{aligned}
\left|\int_\Omega b(x,u_t)z\,dx\right|
=\left|\int_\Omega |u_t|^{m(x)-2}u_tz\,dx\right|\le\varepsilon C\|\nabla z\|_2^2
+C_\varepsilon\int_\Omega|u_t|^{m(x)}\,dx.
\end{aligned}
\end{equation}

By \eqref{eq66}, Young's inequality, and
\eqref{eq71},
\[
\int_\Omega f(x,u)u\,dx
\le\frac{\ell}{2}\|\nabla u\|_2^2,
\]
\[
|(\nabla u,\nabla w)|
\le\frac\ell8\|\nabla u\|_2^2
+\frac2\ell\|\nabla w\|_2^2,
\]
\[
\left|\int_\Omega b(x,u_t)u\,dx\right|
\le\frac\ell8\|\nabla u\|_2^2
+C\int_\Omega|u_t|^{m(x)}\,dx.
\]
Combining these estimates with \eqref{eq44} and using
$1-\int_0^t g(s)\,ds\ge\ell$ proves
\eqref{eq69}.
\end{proof}

\begin{lemma}\label{lem:secondmult}
Under the hypotheses of Lemma~\ref{lem:firstmult},
fix $t_0>0$ and set
$\Gamma_0=\int_0^{t_0}g(s)\,ds>0$.
For every $\delta>0$ and almost every $t\ge t_0$,
\begin{equation}\label{eq72}
\begin{aligned}
\frac d{dt}[-(u_t,w)]
\le{}&-\frac{\Gamma_0}{2}\|u_t\|_2^2
+\delta\|\nabla u\|_2^2
+C_\delta\|\nabla w\|_2^2\\
&+C_\delta\int_\Omega|u_t|^{m(x)}\,dx+C_\delta\int_0^t[-g'(t-s)]
\|\nabla u(t)-\nabla u(s)\|_2^2\,ds,
\end{aligned}
\end{equation}
where $C_\delta$ may depend on $t_0$ but is
independent of time.
\end{lemma}

\begin{proof}
Since $w\in C([0,T];V)\cap L^{m(x)}(Q_T)$ and
$w_t\in C([0,T];H)$, time regularization justifies
the product rule and spatial integration by parts,
yielding
\begin{equation}
\label{eq73}
\begin{split}
\frac d{dt}[-(u_t,w)]
={}&\left(1-\int_0^t g(s)\,ds\right)(\nabla u,\nabla w)
+\|\nabla w\|_2^2\\
&+\int_\Omega b(x,u_t)w\,dx
-\int_\Omega f(x,u)w\,dx
-\left(\int_0^t g(s)\,ds\right)\|u_t\|_2^2\\
&-\left(u_t,\int_0^t g'(t-s)[u(t)-u(s)]\,ds\right)
\end{split}
\end{equation}
for almost every $t$.

By \eqref{eq22}, for $0<|s|\le1$,
\[
|f(x,s)|\le |s|\,|s|^{p_--2}|\log|s||
\le\frac{|s|}{e(p_--2)},
\]
and for $|s|>1$,
$|f(x,s)|\le(e\sigma)^{-1}|s|^{p_++\sigma-1}$.
Since $(p_++\sigma-1)(2^*)'<2^*$, Sobolev embedding
and \eqref{eq68} give
\begin{align*}
\|f(x,u)\|_{(2^*)'}
&\le C\bigl(\|u\|_{(2^*)'}+
\|u\|_{(p_++\sigma-1)(2^*)'}^{p_++\sigma-1}\bigr)\\
&\le C\bigl(\|\nabla u\|_2+
\|\nabla u\|_2^{p_++\sigma-1}\bigr)
\le C(1+R^{p_++\sigma-2})\|\nabla u\|_2.
\end{align*}
Consequently, for every $z\in V$,
\[
|\langle f(x,u),z\rangle|
\le\|f(x,u)\|_{(2^*)'}\|z\|_{2^*}
\le C\|\nabla u\|_2\|z\|_V.
\]
Taking the supremum over $\|z\|_V\le1$, we obtain
\begin{equation}\label{eq74}
\|f(x,u)\|_{V^*}\le C\|\nabla u\|_2.
\end{equation}

For $t\ge t_0$, we have
\[
-\left(\int_0^t g(s)\,ds\right)\|u_t(t)\|_2^2
\le
-\left(\int_0^{t_0}g(s)\,ds\right)\|u_t(t)\|_2^2
=-\Gamma_0\|u_t(t)\|_2^2.
\]
Using \eqref{eq74}, Young's inequality, \eqref{eq71}, the
Cauchy--Schwarz inequality, and Poincar\'e's inequality, we obtain, for every
$\delta>0$,
\[\left|
\left(1-\int_0^t g(s)\,ds\right)(\nabla u,\nabla w)
\right|
+\left|\int_\Omega f(x,u)w\,dx\right|
\le\delta\|\nabla u\|_2^2+C_\delta\|\nabla w\|_2^2,\]
\[\left|\int_\Omega b(x,u_t)w\,dx\right|
\le \delta\|\nabla w\|_2^2
+C_\delta\int_\Omega|u_t|^{m(x)}\,dx,\]
\[
\begin{aligned}
\left|
\left(u_t,\int_0^t g'(t-s)[u(t)-u(s)]\,ds\right)
\right|
\le\frac{\Gamma_0}{2}\|u_t\|_2^2
+\frac{C_p^2g(0)}{2\Gamma_0}
\int_0^t[-g'(t-s)]
\|\nabla u(t)-\nabla u(s)\|_2^2\,ds.
\end{aligned}
\]
Combining these estimates with \eqref{eq73} proves \eqref{eq72}.
\end{proof}

\begin{lemma}\label{lem:basicL}
Under the hypotheses of Lemmas~\ref{lem:firstmult}
and~\ref{lem:secondmult}, fix $t_0>0$.
There exist constants $N_1,N_2,C_L,c,C>0$ such that
\[
L(t)=N_1E(t)+(u(t),u_t(t))-N_2(u_t(t),w(t))
\]
satisfies
\begin{equation}\label{eq75}
\begin{aligned}
E(t)\le L(t)\le C_LE(t),\quad t\ge t_0
\end{aligned}
\end{equation}
and
\begin{equation}\label{eq76}
L'(t)\le-cE(t)
+C\int_0^t g(t-s)
\|\nabla u(t)-\nabla u(s)\|_2^2\,ds, \quad \text{a.e. } t\ge t_0.
\end{equation}
\end{lemma}

\begin{proof}

To establish equivalence with the energy, Cauchy--Schwarz inequality,
Poincar\'e's inequality, \eqref{eq14}, and
\eqref{eq70} yield
\[
\begin{aligned}
|(u,u_t)|
&\le\frac{C_p}{2}
\bigl(\|\nabla u\|_2^2+\|u_t\|_2^2\bigr)
\le C_5E(t),\\
|(u_t,w)|
&\le\frac{C_p}{2}
\bigl(\|u_t\|_2^2+\|\nabla w\|_2^2\bigr)
\le C_6E(t).
\end{aligned}
\]
Increasing $N_1$, if necessary, so that
$N_1\ge C_5+N_2C_6+1$, we obtain
\[
\bigl(N_1-C_5-N_2C_6\bigr)E(t)
\le L(t)
\le\bigl(N_1+C_5+N_2C_6\bigr)E(t).
\]
This proves \eqref{eq75} with $C_L=N_1+C_5+N_2C_6$.

Set $N_2=4/\Gamma_0$ and choose $\delta=\ell/(8N_2)$
in \eqref{eq72}. Adding \eqref{eq69}
and $N_2$ times \eqref{eq72}, we obtain
\[
\begin{aligned}
\frac{d}{dt}\bigl[(u,u_t)-N_2(u_t,w)\bigr]
\le{}&
-\|u_t\|_2^2-\frac{\ell}{8}\|\nabla u\|_2^2
+C_7\|\nabla w\|_2^2+C_8\int_\Omega|u_t|^{m(x)}\,dx\\
&+C_9\int_0^t[-g'(t-s)]
\|\nabla u(t)-\nabla u(s)\|_2^2\,ds.
\end{aligned}
\]
Choosing $N_1\ge\max\{C_8,2C_9\}$, the energy identity
\eqref{eq41} absorbs the last two terms and gives
the differential inequality
\begin{equation}
\label{eq77}
L'(t)\le
-c_1\bigl(\|u_t(t)\|_2^2+\|\nabla u(t)\|_2^2\bigr)
+C_7\|\nabla w(t)\|_2^2
\end{equation}
with $c_1=\min\{1,\ell/8\}$. \eqref{eq14} implies
\[
\|u_t(t)\|_2^2+\|\nabla u(t)\|_2^2
\ge\frac{E(t)}{C_E}
-\int_0^t g(t-s)
\|\nabla u(t)-\nabla u(s)\|_2^2\,ds.
\]
Therefore, it follows from \eqref{eq70} and \eqref{eq77} that 
\[
\begin{aligned}
L'(t)\le
-\frac{c_1}{C_E}E(t)
+\bigl[c_1+C_7(1-\ell)\bigr]
\int_0^t g(t-s)
\|\nabla u(t)-\nabla u(s)\|_2^2\,ds,
\end{aligned}
\]
which proves \eqref{eq76}.
\end{proof}

\begin{lemma}[Integrability of the energy]\label{lem:energyintegrability}
Assume the hypotheses of Lemma~\ref{lem:basicL} and
$-g'>0$ almost everywhere on $(0,\infty)$.
Then there exists $C>0$, independent of time and of the
individual initial data in the stable region, such that
\begin{equation}\label{eq78}
\int_0^\infty E(t)\,dt\le CE(0),
\quad
E(t)\le\frac{CE(0)}{1+t},
\quad t\ge0.
\end{equation}
\end{lemma}

\begin{proof}
Fix $t_0>0$ and use the constants from
Lemma~\ref{lem:basicL}. For $\delta>0$, set
\[
c_\delta=\int_0^\infty
\frac{g(s)^2}{\delta g(s)-g'(s)}\,ds.
\]
Since $\delta g(s)-g'(s)\ge\delta g(s)>0$, we have
\[
c_\delta\le\frac{1}{\delta}
\int_0^\infty g(s)\,ds
=\frac{1-\ell}{\delta}<\infty,\quad 0\le
\frac{\delta g(s)^2}{\delta g(s)-g'(s)}
\le g(s).
\]
The assumption $-g'(s)>0$ almost everywhere implies
$
\frac{\delta g(s)^2}{\delta g(s)-g'(s)}
\to 0$ for almost every $s>0$  as  $\delta\to 0$. 
Since $g\in L^1(0,\infty)$, the dominated convergence
theorem yields
\begin{equation}\label{eq79}
\delta c_\delta
=\int_0^\infty
\frac{\delta g(s)^2}{\delta g(s)-g'(s)}\,ds
\to 0
\quad\text{as }\delta\to 0.
\end{equation}

By the definition of $w$ and 
Cauchy--Schwarz inequality,
\begin{equation}
\label{eq80}
\begin{split}
\|\nabla w(t)\|_2^2
&\le
\left(\int_0^t
\frac{g(t-s)^2}{\delta g(t-s)-g'(t-s)}\,ds\right)
\int_0^t[\delta g(t-s)-g'(t-s)]
\|\nabla u(t)-\nabla u(s)\|_2^2\,ds\\
&\le c_\delta\int_0^t
[\delta g(t-s)-g'(t-s)]
\|\nabla u(t)-\nabla u(s)\|_2^2\,ds.
\end{split}
\end{equation}

Define the nonnegative tail convolution
\[
A_g(t)=\int_0^t
\left(\int_{t-s}^\infty g(r)\,dr\right)
\|\nabla u(s)\|_2^2\,ds.
\]
Then
\begin{equation}
\label{eq81}
0\le A_g(t)
\le(1-\ell)\int_0^t\|\nabla u(s)\|_2^2\,ds.
\end{equation}
Using $\|a-b\|_2^2\le2\|a\|_2^2+2\|b\|_2^2$, we obtain
\begin{equation}\label{eq82}
\begin{split}
A_g'(t)&=(1-\ell)\|\nabla u(t)\|_2^2
-\int_0^t g(t-s)\|\nabla u(s)\|_2^2\,ds\\
&\le
-\frac12\int_0^t g(t-s)
\|\nabla u(t)-\nabla u(s)\|_2^2\,ds
+C_{10}\|\nabla u(t)\|_2^2
\end{split}
\end{equation}
with $C_{10}=2(1-\ell).$ 
Combining \eqref{eq82} with \eqref{eq41} and \eqref{eq80} leads to
\begin{equation}
\label{eq83}
\begin{split}
c_\delta(\delta A_g'+E')
&\le-\frac{c_\delta}{2}\int_0^t
[\delta g(t-s)-g'(t-s)]
\|\nabla u(t)-\nabla u(s)\|_2^2\,ds
+C_{10}\delta c_\delta\|\nabla u(t)\|_2^2\\
&\le-\frac12\|\nabla w(t)\|_2^2
+C_{10}\delta c_\delta\|\nabla u(t)\|_2^2.
\end{split}
\end{equation}

Set $\kappa=2C_7+2$. By \eqref{eq79}, choose
$\delta>0$ sufficiently small and then $\rho>0$ such that
\[
\kappa C_{10}\delta c_\delta\le\frac{c_1}{2},
\quad
\rho C_{10}\le\frac{c_1}{4}.
\]
For
\[
\mathcal Y(t)
=L(t)+\kappa c_\delta E(t)
+(\kappa\delta c_\delta+\rho)A_g(t),
\]
then it follows from \eqref{eq77}, \eqref{eq82} and \eqref{eq83} that 
\begin{equation}
\label{eq84}
\begin{split}
\mathcal Y'(t)
&\le-c_1\|u_t(t)\|_2^2
-\bigl(c_1-\kappa C_{10}\delta c_\delta-\rho C_{10}\bigr)
 \|\nabla u(t)\|_2^2
+\left(C_7-\frac{\kappa}{2}\right)\|\nabla w(t)\|_2^2\\
&\quad-\frac{\rho}{2}\int_0^t g(t-s)
\|\nabla u(t)-\nabla u(s)\|_2^2\,ds\\
&\le-c_1\|u_t(t)\|_2^2
-\frac{c_1}{4}\|\nabla u(t)\|_2^2
-\|\nabla w(t)\|_2^2\\
&\quad-\frac{\rho}{2}\int_0^t g(t-s)
\|\nabla u(t)-\nabla u(s)\|_2^2\,ds
\le-c_2E(t)
\end{split}
\end{equation}
for almost every $t\ge t_0$, where
$
c_2=\frac1{C_E}
\min\left\{\frac{c_1}{4},\frac{\rho}{2}\right\}>0
$
by \eqref{eq14}.

By \eqref{eq75}, \eqref{eq14},
\eqref{eq41}, and \eqref{eq81}, one has 
\[
\mathcal Y(t)\ge L(t)\ge E(t),
\quad
A_g(t_0)\le\frac{(1-\ell)t_0}{c_E}E(0),
\quad
\mathcal Y(t_0)\le CE(0).
\]
Integrating \eqref{eq84} results in 
\[
c_2\int_{t_0}^T E(t)\,dt
\le\mathcal Y(t_0)-\mathcal Y(T)
\le CE(0).
\]
Letting $T\to\infty$ and using \eqref{eq41} on
$[0,t_0]$ proves the first estimate in \eqref{eq78}.
Since 
\[
tE(t)\le\int_0^t E(s)\,ds
\le\int_0^\infty E(s)\,ds\le CE(0), 
\]
adding $E(t)\le E(0)$ shows
\[
E(t)\le\frac{(1+C)E(0)}{1+t},
\quad t\ge0,
\]
which proves the second estimate in \eqref{eq78}.This completes the proof.
\end{proof}

\subsection{Proof of Theorem~\ref{thm:convex}} The proof is divided into two cases.

\textbf{Case 1: $\mathcal H(s)=ks$.}
Since $\xi$ is nonincreasing, \eqref{eq15} and
\eqref{eq41} yield
\begin{equation}\label{eq85}
\begin{aligned}
\xi(t)\int_0^t g(t-s)
\|\nabla u(t)-\nabla u(s)\|_2^2\,ds
&\le \frac1k\int_0^t[-g'(t-s)]
\|\nabla u(t)-\nabla u(s)\|_2^2\,ds\\
&\le -\frac2kE'(t).
\end{aligned}
\end{equation}
Fix $t_0>0$ and take $L(t)$ from Lemma~\ref{lem:basicL}.
Set $\mathcal F(t)=\xi(t)L(t)+C_{11}E(t)$, where
$C_{11}\ge 2C/k+1$ and $C$ is the constant in
\eqref{eq76}. By \eqref{eq75} and the
monotonicity of $\xi$,
\begin{equation}\label{eq86}
C_{11}E(t)\le \mathcal F(t)\le C_{12} E(t)
\end{equation}
with $C_{12}:=\xi(0)C_L+C_{11}.$ 
Using $\xi'(t)\le0$, $L(t)\ge0$, and
\eqref{eq41}, \eqref{eq76},
\eqref{eq85}, and \eqref{eq86},
we obtain, for almost every $t\ge t_0$,
\begin{equation}\label{eq87}
\begin{aligned}
\mathcal F'(t)
&=\xi'(t)L(t)+\xi(t)L'(t)+C_{11}E'(t)\\
&\le -c\xi(t)E(t)
+\left(C_{11}-\frac{2C}{k}\right)E'(t)\\
&\le -c\xi(t)E(t)
\le -\frac{c}{C_{12}}\xi(t)\mathcal F(t).
\end{aligned}
\end{equation}
Integrating \eqref{eq87} and using
\eqref{eq86} and \eqref{eq41}, we find
\[
E(t)\le \frac{\mathcal F(t)}{C_{11}}
\le \frac{\mathcal F(t_0)}{C_{11}}
\exp\left(-\frac{c}{C_{12}}\int_{t_0}^t\xi(s)\,ds\right)
\le \frac{C_{12}}{C_{11}}E(0)
\exp\left(-\frac{c}{C_{12}}\int_{t_0}^t\xi(s)\,ds\right),
\quad t\ge t_0.
\]
Absorbing the fixed factor
$\exp\bigl(\frac{c}{C_{12}}\int_0^{t_0}\xi(s)\,ds\bigr)$
into the multiplicative constant and using
$E(t)\le E(0)$ from \eqref{eq41} on $[0,t_0]$,
we conclude, after renaming the constants, that
\[
E(t)\le CE(0)
\exp\left(-c\int_0^t\xi(s)\,ds\right),
\quad t\ge0.
\]

\textbf{Case 2: $\mathcal H$ is nonlinear.}
By \eqref{eq15}, $-g'>0$ almost everywhere.
Combining \eqref{eq78} with
\eqref{eq14}, we obtain
\begin{equation}\label{eq88}
\begin{split}
\int_0^t\|\nabla u(t)-\nabla u(s)\|_2^2\,ds
&\le  2t\|\nabla u(t)\|_2^2
  +2\int_0^t\|\nabla u(s)\|_2^2\,ds\\
 & \le \frac2{c_E}
\left(tE(t)+\int_0^tE(s)\,ds\right)
\le C_{13}E(0).
\end{split}
\end{equation}

For almost every $s\in(0,t_*)$, the monotonicity of
$\xi$ and $g$, together with \eqref{eq15}, gives
\[
-g'(s)\ge\xi(s)\mathcal H(g(s))
\ge\xi(t_*)\min_{g(t_*)\le r\le g(0)}\mathcal H(r)>0.
\]
Since $\mathcal H$ is continuous and positive on
$[g(t_*),g(0)]$, we may set
\[
K_*:=
\frac{g(0)}
{\xi(t_*)\min_{g(t_*)\le r\le g(0)}\mathcal H(r)}>0,
\]
so that $g(s)\le K_*[-g'(s)]$ almost everywhere
on $(0,t_*)$. Hence, by \eqref{eq41},
\begin{equation}\label{eq89}
\begin{aligned}
&\int_0^{t_*}g(s)
\|\nabla u(t)-\nabla u(t-s)\|_2^2\,ds\\
&\quad\le K_*\int_0^t[-g'(s)]
\|\nabla u(t)-\nabla u(t-s)\|_2^2\,ds
\le-2K_*E'(t),
\end{aligned}
\end{equation}
for almost every $t\ge t_*$.

Take $L(t)$ from Lemma~\ref{lem:basicL} with $t_0=t_*$
and set $\mathcal F(t)=L(t)+2CK_*E(t)$, where $C$ is the
constant in \eqref{eq76}.
By \eqref{eq75},
\begin{equation}\label{eq90}
E(t)\le\mathcal F(t)\le C_{14}E(t).
\end{equation}
Moreover, \eqref{eq76} and
\eqref{eq89} yield
\begin{equation}\label{eq91}
\mathcal F'(t)
=L'(t)+2CK_*E'(t)
\le-cE(t)
+C\int_{t_*}^t g(s)
\|\nabla u(t)-\nabla u(t-s)\|_2^2\,ds.
\end{equation}

Set $B_0=C_{13}E(0)$ and
\[
I(t):=\int_{t_*}^t g(s)
\|\nabla u(t)-\nabla u(t-s)\|_2^2\,ds.
\]
By \eqref{eq88},
\[
d\mu_t(s):=
\frac{\|\nabla u(t)-\nabla u(t-s)\|_2^2}{B_0}\,ds,
\quad
\mu_t([t_*,t])\le1.
\]
Since $\mathcal H(0)=0$ and $g(s)\le r_0$ for $s\ge t_*$,
Jensen's inequality,
together with \eqref{eq15}, gives
\begin{equation}\label{eq92}
\begin{aligned}
B_0\mathcal H\left(\frac{I(t)}{B_0}\right)
&\le\int_{t_*}^t\mathcal H(g(s))
\|\nabla u(t)-\nabla u(t-s)\|_2^2\,ds\\
&\le\frac1{\xi(t)}
\int_{t_*}^t[-g'(s)]
\|\nabla u(t)-\nabla u(t-s)\|_2^2\,ds.
\end{aligned}
\end{equation}

Choose $0<\varepsilon<r_0$ and set
$q(t)=\varepsilon E(t)/E(0)$. We first work where
$E(t)>0$, so that $0<q(t)\le\varepsilon<r_0$.
By convexity and $\mathcal H\ge0$, for $0<q(t)<r_0,~0\le y(t)\le r_0$,
\begin{equation}\label{eq93}
\mathcal H'(q(t))y(t)
\le\mathcal H(y(t))-\mathcal H(q(t))+q(t)\mathcal H'(q(t))
\le\mathcal H(y(t))+q(t)\mathcal H'(q(t)).
\end{equation}
Applying \eqref{eq93} with $y(t)=I(t)/B_0$,
and using \eqref{eq92}, \eqref{eq41},
and $B_0q(t)=C_{13}\varepsilon E(t)$, we obtain
\begin{equation}\label{eq94}
\begin{aligned}
\xi(t)\mathcal H'(q(t))I(t)
&\le \xi(t)B_0\mathcal H\left(\frac{I(t)}{B_0}\right)
+\xi(t)B_0q(t)\mathcal H'(q(t))\\
&\le-2E'(t)
+C_{13}\varepsilon\xi(t)E(t)\mathcal H'(q(t)).
\end{aligned}
\end{equation}

Define
\begin{equation*}
\mathcal G(t)
=\xi(t)\mathcal H'(q(t))\mathcal F(t)+C_{15}E(t).
\end{equation*}
Since $\xi'\le0$, $q'=\varepsilon E'/E(0)\le0$,
$\mathcal H''\ge0$, and $\mathcal F\ge0$, it follows from \eqref{eq91}, and
\eqref{eq94} that 
\begin{equation*}
\begin{aligned}
\mathcal G'(t)&
=\xi'(t)\mathcal H'(q(t))\mathcal F(t)
+\xi(t)\mathcal H''(q(t))q'(t)\mathcal F(t)
+\xi(t)\mathcal H'(q(t))\mathcal F'(t)
+C_{15}E'(t)\\
&\le\xi(t)\mathcal H'(q(t))\mathcal F'(t)+C_{15}E'(t)\\
&\le-c\xi(t)E(t)\mathcal H'(q(t))
+C\xi(t)\mathcal H'(q(t))I(t)+C_{15}E'(t)\\
&\le-(c-CC_{13}\varepsilon)\xi(t)E(t)\mathcal H'(q(t))
+(C_{15}-2C)E'(t).
\end{aligned}
\end{equation*}
Choose $\varepsilon$ sufficiently small that
$CC_{13}\varepsilon\le c/2$, and take $C_{15}\ge2C+1$.
Then
\begin{equation}\label{eq95}
\mathcal G'(t)
\le-\frac c2\xi(t)E(t)
\mathcal H'\left(\frac{\varepsilon E(t)}{E(0)}\right).
\end{equation}
Moreover, \eqref{eq90},
$\xi(t)\le\xi(0)$, and
$\mathcal H'(q(t))\le\mathcal H'(\varepsilon)$ imply
\begin{equation}\label{eq96}
C_{15}E(t)\le\mathcal G(t)\le C_{16}E(t)
\end{equation}
with $
C_{16}:=\xi(0)\mathcal H'(\varepsilon)C_{14}+C_{15}.$

Set $v(t)=\mathcal G(t)/(C_{16}E(0))$.
Then $0<v(t)\le E(t)/E(0)\le1$.
Since $s\mapsto s\mathcal H'(\varepsilon s)$ is
nondecreasing, \eqref{eq95} gives
\begin{equation}\label{eq97}
v'(t)\le-\frac{c}{2C_{16}}\xi(t)
v(t)\mathcal H'(\varepsilon v(t)).
\end{equation}
For $z(t)=\varepsilon v(t)$, the definition
\eqref{eq16} and \eqref{eq97} therefore yields
\begin{equation*}
\frac d{dt}\Psi(z(t))
=-\frac{z'(t)}{z(t)\mathcal H'(z(t))}
\ge\frac{c}{2C_{16}}\xi(t).
\end{equation*}
Integrating from $t_*$ to $t$ and using
$\Psi(z(t_*))\ge0$ due to \(0<z(t_*)\le\varepsilon<r_0\), we obtain
\begin{equation*}
\Psi(z(t))\ge\Psi(z(t_*))
+\frac{c}{2C_{16}}\int_{t_*}^t\xi(s)\,ds
\ge
\frac{c}{2C_{16}}\int_{t_*}^t\xi(s)\,ds.
\end{equation*}
Since $\Psi^{-1}$ is decreasing,
\eqref{eq96} gives
\[
\begin{aligned}
E(t)
\le\frac{C_{16}}{C_{15}\varepsilon}E(0)z(t)
\le\frac{C_{16}}{C_{15}\varepsilon}E(0)
\Psi^{-1}\left(
\frac{c}{2C_{16}}\int_{t_*}^t\xi(s)\,ds
\right),
\quad t\ge t_*.
\end{aligned}
\]
This proves \eqref{eq17}.
If $E$ reaches zero, it remains zero by
\eqref{eq41}, so the estimate continues to hold.

Finally, assume $\int_0^\infty\xi(s)\,ds=\infty$.
Choose $t_1\ge t_*$ such that
$\int_0^{t_1}\xi(s)\,ds\ge2\int_0^{t_*}\xi(s)\,ds$.
Then
\begin{equation*}
\int_{t_*}^t\xi(s)\,ds
=\int_0^t\xi(s)\,ds-\int_0^{t_*}\xi(s)\,ds
\ge\frac12\int_0^t\xi(s)\,ds,
\quad t\ge t_1.
\end{equation*}
Since $\Psi^{-1}$ is decreasing,
\eqref{eq17} yields
\[
E(t)\le CE(0)
\Psi^{-1}\left(\frac c2\int_0^t\xi(s)\,ds\right),
\quad t\ge t_1.
\]
Renaming $c/2$ as $c$, and increasing $C$ so that
\[
C\ge
\left[\Psi^{-1}\left(c\int_0^{t_1}\xi(s)\,ds\right)\right]^{-1},
\]
we also obtain, for $0\le t\le t_1$,
\[
E(t)\le E(0)
\le CE(0)\Psi^{-1}\left(c\int_0^t\xi(s)\,ds\right).
\]
Thus \eqref{eq18} holds for all $t\ge0$.

\subsection{Proof of Corollary~\ref{thm:decay}}
 The case $q=1$ of
Corollary~\ref{thm:decay} is direct. 

For $1<q<2$, take $\mathcal H(s)=s^q$ and $r_0=g(0)$ in
Theorem~\ref{thm:convex}. Direct integration shows
\[
\Psi(s)=\frac{s^{1-q}-r_0^{1-q}}{q(q-1)},\quad
\Psi^{-1}(y)
=\bigl[r_0^{1-q}+q(q-1)y\bigr]^{-1/(q-1)}.
\]
Thus \eqref{eq17} implies
\[
E(t)\le CE(0)
\left(1+c(q-1)\int_{t_*}^t\xi(s)\,ds\right)^{-1/(q-1)},
\quad t\ge t_*.
\]
For $t\ge t_*$, write $I_*=\int_0^{t_*}\xi(s)\,ds$. Since
\[
1+c(q-1)\int_0^t\xi(s)\,ds
\le[1+c(q-1)I_*]
\left[1+c(q-1)\int_{t_*}^t\xi(s)\,ds\right],
\]
the finite shift is absorbed into the multiplicative constant.
Increasing that constant on $[0,t_*]$ proves \eqref{eq19}
without requiring $\int_0^\infty\xi=\infty$.

\section{Decay beyond the Sobolev damping range}\label{sec:superproof}

Assume $u_0\in L^{m(x)}(\Omega)$. 
For almost every $x\in\Omega$, absolute continuity
in time and H\"older's inequality give
\[
\begin{aligned}
|u(x,t)|^{m(x)}
&=\left|u_0(x)+\int_0^t u_t(x,s)\,ds\right|^{m(x)}\\
&\le2^{m_+-1}\left[
|u_0(x)|^{m(x)}
+t^{m(x)-1}\int_0^t|u_t(x,s)|^{m(x)}\,ds
\right].
\end{aligned}
\]
Dividing by $(1+t)^{m(x)-1}$, integrating over
$\Omega$, and using \eqref{eq7} and
$E(t)\ge0$ from \eqref{eq14}, we obtain
\begin{equation}\label{eq98}
\begin{aligned}
\int_\Omega
\frac{|u(x,t)|^{m(x)}}{(1+t)^{m(x)-1}}\,dx
&\le2^{m_+-1}\left[
\int_\Omega|u_0(x)|^{m(x)}\,dx
+\int_0^t\int_\Omega|u_t|^{m(x)}\,dx\,ds
\right]\\
&\le2^{m_+-1}\left[
\int_\Omega|u_0(x)|^{m(x)}\,dx+E(0)
\right]
=:K_0.
\end{aligned}
\end{equation}
Consequently, for every finite $T>0$,
\[
\int_{Q_T}|u|^{m(x)}\,dx\,dt
\le T(1+T)^{m_+-1}K_0<\infty.
\]
Thus $u\in L^{m(x)}(Q_T)$, while
$b(x,u_t)\in L^{m'(x)}(Q_T)$ by the damping estimate.
Together with the regularity in
Definition~\ref{def:energy}, these properties justify
testing the weak equation with $u$ by temporal
regularization.

With $\gamma$ as in \eqref{eq20}, set
\[
W(t)=\frac{E(t)^\gamma}{1+t}.
\]
We first show that $E(t)>0$ for every finite $t\ge0$.
Suppose, to the contrary, that $E(t_1)=0$ for some
$t_1>0$. By \eqref{eq41}, $E$ is nonincreasing, while
\eqref{eq14} gives $E(t)\ge0$. Hence $E(t)=0$ for all
$t\ge t_1$. In particular, \eqref{eq14} implies
$\widehat E(t_1)=0$, and therefore
\[
u_t(t_1)=0,\quad \nabla u(t_1)=0,\quad
\int_0^{t_1}g(t_1-s)
\|\nabla u(t_1)-\nabla u(s)\|_2^2\,ds=0.
\]
Since $g>0$, the last equality and $\nabla u(t_1)=0$
imply $\nabla u(s)=0$ for almost every $s\in[0,t_1]$.
The continuity $u\in C([0,t_1];V)$ then yields
$u(s)=0$ in $V$ for every $s\in[0,t_1]$. In particular,
$u_0=u_1=0$, which contradicts $E(0)>0$. Thus
$E(t)>0$ for every finite $t\ge0$.

Consequently, $W$ is locally absolutely continuous on
every finite time interval, and \eqref{eq41} gives
\[
W'(t)
=\frac{\gamma E(t)^{\gamma-1}E'(t)}{1+t}
-\frac{E(t)^\gamma}{(1+t)^2}\le0
\quad\text{for almost every }t\ge0.
\]
We may therefore derive the estimates below for arbitrary
finite $T$ and then let $T\to\infty.$

By \eqref{eq98}, the time-regularization
argument leading to \eqref{eq44} remains valid here.
Multiplying that identity by $W(t)$ and integrating by parts
over $[0,T]$, we use \eqref{eq66}, \eqref{eq68},
and the definition of $R$ to obtain
\[
\begin{aligned}
\frac{\ell}{2}\int_0^T W(t)\|\nabla u\|_2^2\,dt
\le{}&
\int_0^T W(t)\|u_t\|_2^2\,dt-[W(t)(u,u_t)]_0^T
+\int_0^T W'(t)(u,u_t)\,dt\\
&+\int_0^T W(t)|(\nabla u,\nabla w)|\,dt
+\int_0^T W(t)\left|\int_\Omega b(x,u_t)u\,dx\right|dt.
\end{aligned}
\]
Young's inequality and \eqref{eq70} lead to
\[
|(\nabla u,\nabla w)|
\le\frac{\ell}{4}\|\nabla u\|_2^2
+\frac{1-\ell}{\ell}\int_0^t g(t-s)
\|\nabla u(t)-\nabla u(s)\|_2^2\,ds.
\]
Moreover, \eqref{eq14} and Poincar\'e's 
inequality yield $|(u,u_t)|\le CE(t)$.
Since $W'(t)\le0$ and $E(t)\le E(0)$,
\[
|[W(t)(u,u_t)]_0^T|
+\int_0^T(-W'(t))|(u,u_t)|\,dt
\le CE(0)^{\gamma+1}.
\]
Absorbing the gradient contribution and applying the upper
bound in \eqref{eq14}, we conclude that
\begin{equation}\label{eq99}
\begin{aligned}
\int_0^T\frac{E(t)^{\gamma+1}}{1+t}\,dt
\le \widetilde{C}\bigg[&
E(0)^{\gamma+1}
+\int_0^T W(t)\|u_t(t)\|_2^2\,dt\\
&+\int_0^T W(t)\int_0^t g(t-s)
\|\nabla u(t)-\nabla u(s)\|_2^2\,ds\,dt\\
&+\int_0^T W(t)
\left|\int_\Omega b(x,u_t)u\,dx\right|dt
\bigg].
\end{aligned}
\end{equation}
Here $\widetilde{C}>0$ is independent of $T$ and is fixed before
choosing the history cutoff.

\textbf{Step 1.} 
Fix $a>0$. We split the memory integral into the
short-memory part, corresponding to $0<s<a$, and the
long-memory part, corresponding to $s\ge a$:
\begin{align*}
&\int_0^T W(t)\int_0^t g(s)
 \|\nabla u(t)-\nabla u(t-s)\|_2^2\,ds\,dt\\
&\quad=
\int_0^T W(t)\int_0^{\min\{a,t\}}g(s)
 \|\nabla u(t)-\nabla u(t-s)\|_2^2\,ds\,dt\\
&\qquad+
\int_a^T W(t)\int_a^t g(s)
 \|\nabla u(t)-\nabla u(t-s)\|_2^2\,ds\,dt,
\quad T>a.
\end{align*}
When $T\le a$, the long-memory part is zero.

For almost every $s\in(0,a)$, Assumption~\ref{assadd}
and the monotonicity of $\xi$ and $g$ give
\[
-g'(s)\ge \xi(s)\mathcal H(g(s))
\ge \xi(a)\min_{g(a)\le r\le g(0)}\mathcal H(r)>0.
\]
Here the minimum is positive because $\mathcal H$ is
continuous and positive on the compact interval
$[g(a),g(0)]$. Hence, with
\[
K_a=
\frac{g(0)}
{\xi(a)\min_{g(a)\le r\le g(0)}\mathcal H(r)},
\]
we have $g(s)\le K_a[-g'(s)]$ for almost every
$s\in(0,a)$. Since $W(t)\le E(0)^\gamma$, the energy
identity~\eqref{eq7} yields
\begin{equation}\label{eq100}
\begin{aligned}
&\int_0^T W(t)\int_0^{\min\{a,t\}}g(s)
\|\nabla u(t)-\nabla u(t-s)\|_2^2\,ds\,dt\\
&\quad\le K_aE(0)^\gamma
\int_0^T\int_0^t[-g'(s)]
\|\nabla u(t)-\nabla u(t-s)\|_2^2\,ds\,dt
\le2K_aE(0)^{\gamma+1}.
\end{aligned}
\end{equation}
In the last inequality, we used the change of variables
$r=t-s$ to identify the double integral with the
memory-dissipation term in~\eqref{eq7}.

For $T>a$, use
\[
\|\nabla u(t)-\nabla u(t-s)\|_2^2
\le2\|\nabla u(t)\|_2^2
 +2\|\nabla u(t-s)\|_2^2.
\]
Tonelli's theorem, the change of variables $r=t-s$ and $W(r+s)\le W(r)$ leads to 
\begin{equation}\label{eq101}
\begin{aligned}
&\int_a^T W(t)\int_a^t g(s)
\|\nabla u(t)-\nabla u(t-s)\|_2^2\,ds\,dt\\
&\quad\le
2\left(\int_a^\infty g(s)\,ds\right)
\int_0^T W(t)\|\nabla u(t)\|_2^2\,dt\\
&\qquad+
2\int_a^T g(s)\int_0^{T-s}
W(r+s)\|\nabla u(r)\|_2^2\,dr\,ds\\
&\quad\le
4\left(\int_a^\infty g(s)\,ds\right)
\int_0^T W(t)\|\nabla u(t)\|_2^2\,dt\\
&\quad\le
\frac4{c_E}\left(\int_a^\infty g(s)\,ds\right)
\int_0^T\frac{E(t)^{\gamma+1}}{1+t}\,dt,
\end{aligned}
\end{equation}
The last inequality follows from
$E(t)\ge c_E\widehat E(t)\ge c_E\|\nabla u(t)\|_2^2$.
For $T\le a$, the left-hand side of~\eqref{eq101} is zero,
so the estimate holds in that case as well.

Since $g\in L^1(0,\infty)$, choose $a$ sufficiently large
that
\begin{equation}\label{eq102}
\frac{4\widetilde C}{c_E}\int_a^\infty g(s)\,ds\le\frac14.
\end{equation}
Then the long-memory contribution in~\eqref{eq99} can be
absorbed into its left-hand side. With this $a$ fixed, the
short-memory term is bounded by
$2K_aE(0)^{\gamma+1}$, independently of $T$.

\textbf{Step 2.}
Splitting $\Omega$ into $\{|u_t|\le1\}$ and
$\{|u_t|>1\}$, and using $2\le m(x)\le m_+$,
H\"older's inequality gives
\begin{equation*}
\|u_t\|_2^2
\le|\Omega|^{1-2/m_+}
\left(\int_\Omega|u_t|^{m(x)}\,dx\right)^{2/m_+}
+\int_\Omega|u_t|^{m(x)}\,dx.
\end{equation*}
Young's inequality with conjugate exponents
$m_+/(m_+-2)$ and $m_+/2$, together with
$\gamma m_+/(m_+-2)\ge\gamma+1$ from
\eqref{eq20} and $E(t)\le E(0)$, yields
\[
E(t)^\gamma\|u_t(t)\|_2^2
\le\varepsilon_1 E(t)^{\gamma+1}
+C_{\varepsilon_1,E(0)}
\int_\Omega|u_t|^{m(x)}\,dx
\]
for every $\varepsilon_1>0$.
Dividing by $1+t$, integrating, and using
\eqref{eq7}, we obtain
\begin{equation}\label{eq103}
\int_0^T W(t)\|u_t(t)\|_2^2\,dt
\le\varepsilon_1\int_0^T
\frac{E(t)^{\gamma+1}}{1+t}\,dt
+C_{\varepsilon_1,E(0)}E(0).
\end{equation}

\textbf{Step 3: }For the damping pairing, Young's inequality with exponents
$m(x)$ and $m'(x)$ gives
\[
\begin{aligned}
E(t)^\gamma|u_t|^{m(x)-1}|u|
\le
\varepsilon_2 E(t)^{\gamma+1}
\frac{|u|^{m(x)}}{(1+t)^{m(x)-1}}+C_{\varepsilon_2}(1+t)
E(t)^{\gamma-1/(m(x)-1)}|u_t|^{m(x)},
\end{aligned}
\]
for every $\varepsilon_2>0$. 
By \eqref{eq20},
$0\le\gamma-1/(m(x)-1)\le\gamma$, so
\[
E(t)^{\gamma-1/(m(x)-1)}
\le\max\{1,E(0)^\gamma\}.
\]
Dividing by $1+t$ and applying
\eqref{eq98} and \eqref{eq7}, we obtain
\begin{equation}\label{eq104}
\begin{aligned}
\int_0^T W(t)
\left|\int_\Omega b(x,u_t)u\,dx\right|dt
\le 
\varepsilon_2 K_0
\int_0^T\frac{E(t)^{\gamma+1}}{1+t}\,dt+C_{\varepsilon_2,E(0)}E(0).
\end{aligned}
\end{equation}

\textbf{Step 4.}
Fix $a$ by \eqref{eq102}, and choose the
small parameters in \eqref{eq103} and
\eqref{eq104} so that
$\widetilde{C}\varepsilon_1\le1/4$ and $\widetilde{C}\varepsilon_2 K_0\le1/4$,
respectively.
Substituting these estimates and
\eqref{eq100}--\eqref{eq101}
into \eqref{eq99},  gives
\[
\int_0^T\frac{E(t)^{\gamma+1}}{1+t}\,dt
\le C_{a,E(0),K_0}
\bigl[E(0)^{\gamma+1}+E(0)\bigr]
=\widehat C E(0)^{\gamma+1},
\]
where
$\widehat C=C_{a,E(0),K_0}[1+E(0)^{-\gamma}]$
may depend on the initial data but is independent
of $T$. Letting $T\to\infty$, we obtain
\begin{equation}\label{eq105}
\int_0^\infty\frac{E(t)^{\gamma+1}}{1+t}\,dt
\le\widehat C E(0)^{\gamma+1}.
\end{equation}

By energy monotonicity from \eqref{eq41}
and the integral estimate \eqref{eq105},
\[
E(t)^{\gamma+1}\log(1+t)
\le\int_0^t\frac{E(s)^{\gamma+1}}{1+s}\,ds
\le\widehat C E(0)^{\gamma+1}.
\]
Adding $E(t)^{\gamma+1}\le E(0)^{\gamma+1}$,
which also follows from \eqref{eq41}, yields
\[
[1+\log(1+t)]E(t)^{\gamma+1}
\le(1+\widehat C)E(0)^{\gamma+1}.
\]
Therefore, 
\[
E(t)\le
(1+\widehat C)^{1/(\gamma+1)}E(0)
[1+\log(1+t)]^{-1/(\gamma+1)},
\]
which proves \eqref{eq21}.

For $m(x)\equiv m>6$, \eqref{eq20} gives
$\gamma=(m-2)/2$, so that $1/(\gamma+1)=2/m$.
Thus \eqref{eq21} becomes
\[
E(t)\le CE(0)[1+\log(1+t)]^{-2/m},
\quad t\ge0,
\]
proving Corollary~\ref{cor:log}.

\vspace{1em}
\noindent\textbf{Competing interest}

The authors have no relevant financial or non-financial interests to disclose.

\vspace{1em}
\noindent\textbf{Data availability}

No data was used for the research described in the article.

\end{document}